\documentclass{article}

\usepackage{amsmath}
\usepackage{amsthm}
\usepackage{paralist}
\usepackage{graphicx}
\usepackage{geometry}
\usepackage{epstopdf}
\usepackage[colorlinks=true]{hyperref}
\usepackage{amssymb,amsmath,multirow,float}
\usepackage{ctable}
\usepackage{centernot}
\usepackage{caption,subcaption}
\usepackage[]{algorithm2e}
\usepackage{authblk}

\RestyleAlgo{ruled}

\hypersetup{urlcolor=blue, citecolor=red}

\newtheorem{theorem}{Theorem}[section]
\newtheorem{corollary}{Corollary}
\newtheorem{lemma}[theorem]{Lemma}
\newtheorem{proposition}{Proposition}

\newtheorem{definition}{Definition}

\newtheorem{example}{Example}

\title{On the convexity for the range set of two quadratic functions}
\author[1,2]{Huu-Quang Nguyen\footnote{quangdhv@gmail.com}}
\author[2]{Ya-Chi Chu\thanks{ycchu.academic@gmail.com}}
\author[2]{Ruey-Lin Sheu\thanks{rsheu@mail.ncku.edu.tw}}

\affil[1]{Institute of Natural Science Education, Vinh University, Vinh, Nghe An, Vietnam.}
\affil[2]{Department of Mathematics, National Cheng Kung University, Tainan, Taiwan. }
\date{}

\begin{document}

\maketitle

\begin{abstract}
Given $n\times n$ symmetric matrices $A$ and $B,$ Dines in 1941 proved that the joint range set $\{(x^TAx,x^TBx)|~x\in\mathbb{R}^n\}$ is always convex. Our paper is concerned with non-homogeneous extension of the Dines theorem for the range set $\mathbf{R}(f,g) = \{\left(f(x),g(x)\right)|~x \in \mathbb{R}^n \},$ $f(x) = x^T A x + 2a^T x + a_0$ and $g(x) = x^T B x + 2b^T x + b_0.$
We show that $\mathbf{R}(f,g)$ is convex if, and only if, any pair of level sets, $\{x\in\mathbb{R}^n|f(x)=\alpha\}$ and $\{x\in\mathbb{R}^n|g(x)=\beta\}$, do not separate each other. With the novel geometric concept about separation, we provide a polynomial-time procedure to practically check whether a given $\mathbf{R}(f,g)$ is convex or not.
\end{abstract}

\section{Introduction} \label{sec:intro}
Given a pair of quadratic functions, $f(x)=x^TAx+2a^Tx+a_0$ and $g(x)=x^TBx+2b^Tx+b_0$, where $A,~B$ are $n\times n$ real symmetric matrices, $a,b\in\mathbb{R}^n$ and $a_0,b_0\in\mathbb{R},$
their joint numerical range, a subset of $\mathbb{R}^2,$ is defined to be
$$
\mathbf{R}(f,g) = \left\{(f(x), g(x))~|~x\in \mathbb{R}^n \right\}\subset\mathbb{R}^2.
$$
In this paper, we are interested in the fundamental mathematical problem:
$$
\text{(P)}: \text{ When, and only when, the joint range } \mathbf{R}(f,g) \text{ is convex?}
$$

The first result regarding (P) was an original paper by Dines \cite[1941]{Dines41}. He showed that the joint range of two homogeneous quadratic forms $\left\{( x^T A x, x^T B x)~|~x \in \mathbb{R}^n \right\}$ is always convex. Though a special case of (P), Yakubovich \cite[1971]{Yakubovich71} used it to prove the classical $\mathcal{S}$-lemma, which later became
an indispensable tool in optimization and the control theory.
The $\mathcal{S}$-lemma asserts that, if $g(x)\le0$ satisfies Slater's
condition, namely, there is an $\overline{x}\in \mathbb{R}^n$ such
that $g(\overline{x}) < 0$, the following two statements are
equivalent:
\begin{itemize}
 \item[](${\rm S_1}$)~~ ($\forall x\in\Bbb R^n$) $~g(x)\le 0~\Longrightarrow~ f(x)\ge 0.$ 
 \item[](${\rm S_2}$)~~   $(\exists\mu\ge0)$~
$f(x) + \mu g(x)\ge0,~\forall x\in \mathbb{R}^n.$ 
\end{itemize}
A number of interesting applications follow from the $\mathcal{S}$-lemma.
In particular, it can be used to show a highly non-trivial result that quadratic program with one quadratic inequality constraint
$$ ({\rm QP1QC})~~~~~ \inf_{x\in \Bbb R^n}\left\{f(x) ~|~ g(x)\le0\right\}$$
always adopts strong duality, while $f$ and $g$ are not necessarily convex.
Survey and extensions of the $\mathcal{S}$-lemma are referred to, for example, Derinkuyu and Pinar \cite[2006]{Derinkuyu-Pinar06}, Pólik and Terlaky \cite[2007]{Polik-Terlaky07}, Tuy and Tuan \cite[2013]{Tuy-Tuan13}, Xia et al. \cite[2016]{Xia-Wang-Sheu16}.

The convexity of the joint numerical range $\mathbf{R}(f,g)$ also allows to reformulate and solve difficult optimization problems. Nguyen et al. \cite[2020]{Quang-Sheu-Xia-Po4} proposed to solve the following special type of quadratic optimization problem with a joint numerical range constraint:
$$
({\rm Po4}) \quad
\begin{array}{cl}
\displaystyle \inf _{(x,z) \in \mathbb{R}^{n} \times \mathbb{R}^2} & F(z) \\
\text { s.t. } & \alpha z_1 + \beta z_2 - \gamma \leq 0 \\
 & z = \left( z_1, z_2 \right) \in \mathbf{R}(f,g)
\end{array}$$
where $\alpha, \beta, \gamma \in \mathbb{R}^m$ and $F(z_1, z_2)$ is a convex quadratic function from $\mathbb{R}^2$ to $\mathbb{R}$. They showed that, if the joint range set $\mathbf{R}(f,g)$ is convex, (Po4) has an equivalent SDP reformulation. Then, some important optimization problems in the literature can be solved, including
%
%
%
\begin{itemize}
\item[-](Not solved efficiently before) Ye and Zhang \cite[2003]{Ye-Zhang16} proposed to minimize the absolute value of a quadratic function over a quadratic constraint:
$$
({\rm AQP}) \quad
\begin{array}{cl}
\displaystyle \inf _{x \in \mathbb{R}^{n}} & \left|x^{T} A x + 2 a^{T} x + a_{0}\right| \\
\text{ s.t. } & x^{T} B x + 2 b^{T} x + b_{0} \leq 0.
\end{array}
$$
It was suggested in \cite[2003]{Ye-Zhang16} to solve the problem $({\rm AQP})$, under strict conditions, by the bisection method, each iteration of which requires to do an SDP. In \cite[2020]{Quang-Sheu-Xia-Po4}, by the help of the convexity of $\mathbf{R}(f,g),$ the problem $({\rm AQP})$ can be resolved completely without any condition.
\item[-] (Not solved before) The quadratic hypersurface intersection problem proposed by Pólik and Terlaky \cite[2007]{Polik-Terlaky07}: given two quadratic surfaces $f(x) = 0$ and $g(x) = 0$, how to determine whether the two quadratic surfaces $f(x) = 0$ and $g(x) = 0$ has intersection without actually computing the intersection? The problem can be reformulated as the following non-linear least square problems:
$$
({\rm QSIC}) \quad
\begin{array}{cl}
\displaystyle \inf _{\left(x^{T}, z_{1}, z_{2}\right)^{T} \in \mathbb{R}^{n+2}} & \left(z_{1}\right)^{2}+\left(z_{2}\right)^{2} \\
\text { s.t. } &  \left\{\begin{array}{l}
f(x)-z_{1}=0 \\
g(x)-z_{2}=0,
\end{array}\right.
\end{array}
$$
which is a type of $
({\rm Po4}).$ In \cite[2020]{Quang-Sheu-Xia-Po4}, Nguyen et al. showed that, if $\mathbf{R}(f,g)$ is convex, $({\rm QSIC})$ can be solved by an SDP. If not, it can be solved directly by elementary analysis.
\item[-] (Not solved efficiently before) The double well potential problems (DWP) in \cite[2017]{DWP-1,DWP-2}:
$$
\begin{array}{rl}
({\rm DWP}) & \displaystyle \inf _{x \in \mathbb{R}^{n}} ~ \frac{1}{2} \left( \frac{1}{2} \|Px - p\|^2 - r \right)^2 + \frac{1}{2} x^T Q x - q^T x
\end{array},
$$
where $Q$ is an $n \times n$ symmetric matrix, $P \neq 0$ is an $m \times n$ matrix, $p \in \mathbb{R}^{m}, r \in \mathbb{R}$, and $q \in \mathbb{R}^n$. The original development for solving (DWP) used elementary (but lengthy) approach. By identifying (DWP) as a special type of (Po4), it can now be solved with just an SDP.
\end{itemize}

Though the characterization of the convexity of $\mathbf{R}(f,g)$ is a useful tool for solving optimization problems, in literature, progresses from Dines' result to the convexity of $\mathbf{R}(f,g)$ for general $f(x)=x^TAx+2a^Tx+a_0$ and $g(x)=x^TBx+2b^Tx+b_0$ have been very slow. The Dines theorem becomes invalid when either $f$ or $g$ or both adopt linear terms. Here is an example with configurations for easy understanding.
\begin{example} \label{ex:linear_term}
\renewcommand*{\arraystretch}{1}
Let $f(x,y) = -x^2+y^2$ and $g(x) = -2x^2+2y^2+4x-2y$ where $g(x)$ has linear terms. In this example, $$A =
\left[ \begin{array}{r r}
-1 & 0 \\
 0 & 1
\end{array} \right]~ \text{ and }~ B =
\left[ \begin{array}{r r}
-2 & 0 \\
 0 & 2
\end{array} \right].$$
Fig. \ref{fig:add_linear_homo} shows that $\left\{( x^T A x, x^T B x)~|~x \in \mathbb{R}^n \right\}$ is a straight line and hence convex. Fig. \ref{fig:add_linear_nonhomo} gives the graph of $\mathbf{R}(f,g)$, which is apparently non-convex.
\end{example}

\begin{figure}
\centering
\begin{subfigure}[c]{0.49\textwidth}
	\centering
	\includegraphics[width=\linewidth]{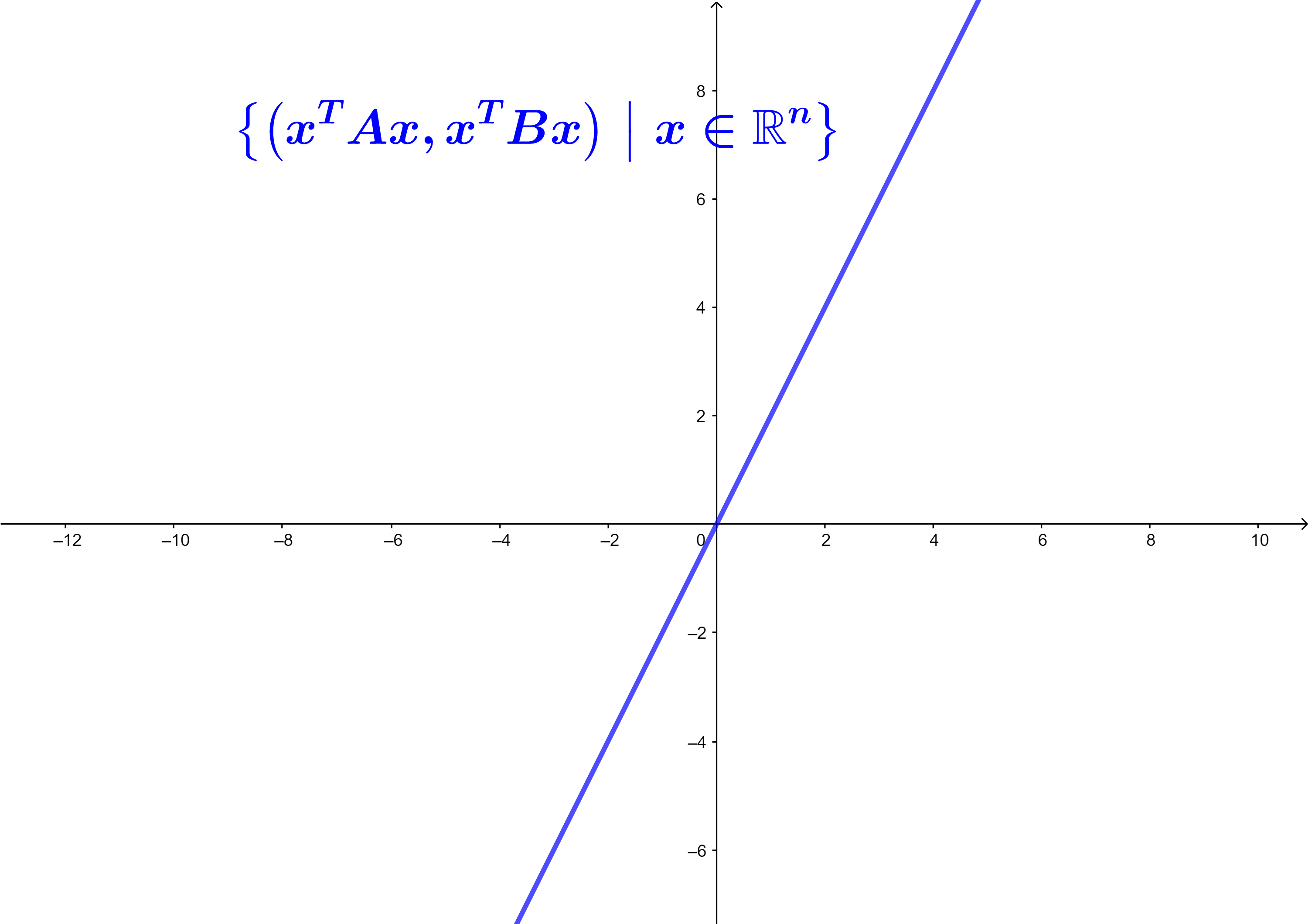}
	\caption{The set $\left\{( x^T A x, x^T B x)~|~x \in \mathbb{R}^n \right\}$ in Example \ref{ex:linear_term} is a straight line.}
	\label{fig:add_linear_homo}
\end{subfigure}
\hfill
\begin{subfigure}[c]{0.49\textwidth}
	\centering
	\includegraphics[width=\linewidth]{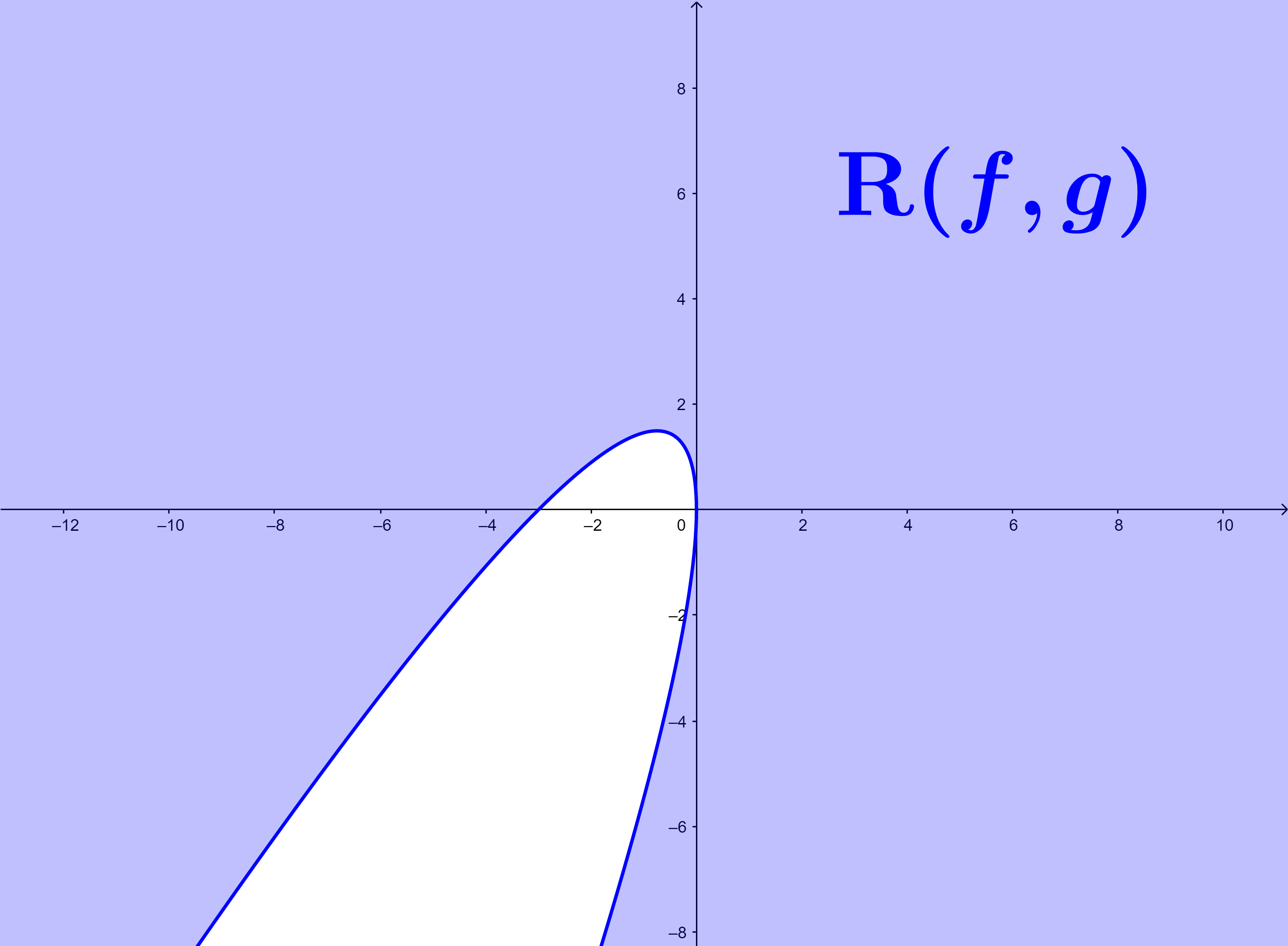}
	\caption{The shaded region is the joint range $\mathbf{R}(f,g)$ in Example \ref{ex:linear_term}.}
	\label{fig:add_linear_nonhomo}
\end{subfigure}
\caption{The graph corresponds to Example \ref{ex:linear_term}.}
\label{fig:add_linear_terms}
\end{figure}

For a long period of time, the best generalization of the Dines theorem has been Polyak's sufficient condition \cite[1998]{Polyak98} (The fourth row in Table \ref{tab:results_summary}). 
It was not until 2016 that Flores-Bazán and Opazo \cite[2016]{Bazan-Opazo16} (last row in Table \ref{tab:results_summary}) completely characterized the convexity of $\mathbf{R}(f,g)$ with a set of necessary and sufficient conditions. Over a period of 75 years from 1941 to 2016, notable results related to (P) include, in chronological order, Brickmen \cite[1961]{Brickmen61}, an unpublished manuscript by Ramana and Goldman \cite[1995]{Ramana-Goldman95}, and Polyak \cite[1998]{Polyak98}.  Among them, we feel that Brickmen's \cite[1961]{Brickmen61} and Flores-Bazán and Opazo's \cite[2016]{Bazan-Opazo16} results are the most fundamental. They are summarized in Table \ref{tab:results_summary}.

\begin{table}[h!]
\centering
\resizebox{\linewidth}{!}{
\begingroup
\renewcommand*{\arraystretch}{1.75}
\begin{tabular}{ c | m{0.8\linewidth}}
\specialrule{.1em}{.05em}{.05em}
\shortstack{ 1941 \\ (Dines \cite{Dines41})} &  \textbf{(Dines Theorem)}
\begin{center}
$\left\{ \left. \left( x^T A x, x^T B x \right) ~\right|~x \in \mathbb{R}^n \right\}$
\end{center}
is convex. Moreover, if $x^T A x$ and $x^T B x$ has no common zero except for $x=0$, then $\left\{ \left. \left( x^T A x, x^T B x \right) ~\right|~x \in \mathbb{R}^n \right\}$ is either $\mathbb{R}^2$ or an angular sector of angle less than $\pi$.  \\ \hline
\shortstack{ 1961 \\ (Brickmen \cite{Brickmen61})} &
\vspace{0.1cm}
\begin{center}
$\mathbf{K}_{A,B} = \left\{ \left. \left( x^T A x, x^T B x \right) ~\right|~x \in \mathbb{R}^n~,~\|x\|=1 \right\}$
\end{center}
is convex if $n \geq 3$.
\\ \hline
\multirow{2}{*}{\shortstack{ 1995 \\ (Ramana \& Goldman \cite{Ramana-Goldman95}) \\ \textbf{Unpublished}}} & \vspace{0.1cm} \begin{center}
$\mathbf{R}(f,g) = \left\{ \left. \left( f(x), g(x) \right) ~\right|~x \in \mathbb{R}^n \right\}$
\end{center}
is convex if and only if $\mathbf{R}(f,g) = \mathbf{R}(f_H,g_H) + \mathbf{R}(f,g)$, where $f_H(x) = x^T A x$ and $g_H(x) = x^T B x$. \\ \cline{2-2}
 & \multirow{2}{*}{\shortstack{$\mathbf{R}(f,g) = \left\{ \left. \left( f(x), g(x) \right) ~\right|~x \in \mathbb{R}^n \right\}$ \\
is convex if $n \geq 2$ and $\exists~ \alpha, \beta \in \mathbb{R}$ such that \mbox{$\alpha A + \beta B \succ 0$}.}} \\ \cline{1-1}
 &  \\ \cline{2-2}
\shortstack{ 1998 \\ (Polyak \cite{Polyak98})} &
\vspace{0.1cm}
\begin{center}
$\left\{ \left. \left( x^T A x, x^T B x, x^T C x \right) ~\right|~x \in \mathbb{R}^n \right\}$
\end{center}
is convex if $n \geq 3$ and $\exists~ \alpha, \beta, \gamma \in \mathbb{R}$ such that \mbox{$\alpha A + \beta B + \gamma C \succ 0$}.   \\ \cline{2-2}
 & \vspace{0.1cm} \begin{center}
$\left\{ \left. \left( x^T A_1 x, \cdots, x^T A_m x \right) ~\right|~x \in \mathbb{R}^n \right\}$
\end{center}
is convex if $A_1, \cdots, A_m$ commute. \\ \hline
\shortstack{ 2016 \\ (Bazán \& Opazo \cite{Bazan-Opazo16})} &
\vspace{0.1cm}
\begin{center}
$\mathbf{R}(f,g) = \left\{ \left. \left( f(x), g(x) \right) ~\right|~x \in \mathbb{R}^n \right\}$
\end{center}
is convex if and only if $\exists~ d=(d_1,d_2) \in \mathbb{R}^2$, $d \neq 0$, such that the following four conditions hold:
\begin{description}
\item[(C1)] $F_L \left( \mathcal{N}(A) \cap \mathcal{N}(B) \right) = \{0\}$
\item[(C2)] $d_2 A = d_1 B$
\item[(C3)] $-d \in \mathbf{R}(f_H,g_H)$
\item[(C4)] $F_H(u) = -d \implies \left\langle F_L(u), d_{\perp}\right\rangle \neq 0$
\end{description}
where $\mathcal{N}(A)$ and $\mathcal{N}(B)$ denote the null space of $A$ and $B$ respectively, $F_H(x) = \left( f_H(x), g_H(x) \right) = \left( x^T A x , x^T B x \right)$, $F_L(x)=\left( a^T x , b^T x \right)$, and $d_{\perp} = (-d_2, d_1)$. \\ \specialrule{.1em}{.05em}{.05em}
\end{tabular}
\endgroup
}
\caption{Chronological list of notable results related to problem (P)}
\label{tab:results_summary}
\end{table}

Flores-Bazán and Opazo's result \cite[2016]{Bazan-Opazo16} (last row in Table \ref{tab:results_summary}) is considered fundamental by us because they were the first to provide a complete answer to problem (P). Their results rely heavily on an unpublished
manuscript by Ramana and Goldman \cite[1995]{Ramana-Goldman95} (third row in Table \ref{tab:results_summary}). In \cite[1995]{Ramana-Goldman95}, it was shown that
$\mathbf{R}(f,g)$ is convex if and only if the following relation holds:
\begin{equation} \label{eq:Homo+NonHomo-relation}
\mathbf{R}(f,g) = \mathbf{R}(f_H,g_H) + \mathbf{R}(f,g),
\end{equation}
where $f_H(x) = x^T A x,~g_H(x) = x^T B x.$
For example, let $f(x,y,z) = x^2+y^2$ and $g(x,y,z) = -x^2+y^2+z.$
Figure \ref{fig:shifted_nonhomo} shows that $\mathbf{R}(f,g)$ is the right half-plane, while Figure \ref{fig:shifted_homo} is the joint range of the homogeneous part $\mathbf{R}(f_H,g_H)$ which is an angular sector of angle $\pi/2$. In this example, it can be checked that
\begin{equation*} 
\mathbf{R}(f_H,g_H) + \mathbf{R}(f,g) = \bigcup_{y \in \mathbf{R}(f,g)} \mathbf{R}(f_H,g_H)+y \subset \mathbf{R}(f,g).
\end{equation*}
Moreover, there is $\mathbf{R}(f,g)\subset \mathbf{R}(f_H,g_H) + \mathbf{R}(f,g),$ which verifies \eqref{eq:Homo+NonHomo-relation} for this example.

\begin{figure}
\centering
\begin{subfigure}[c]{0.49\textwidth}
	\centering
	\includegraphics[width=\linewidth]{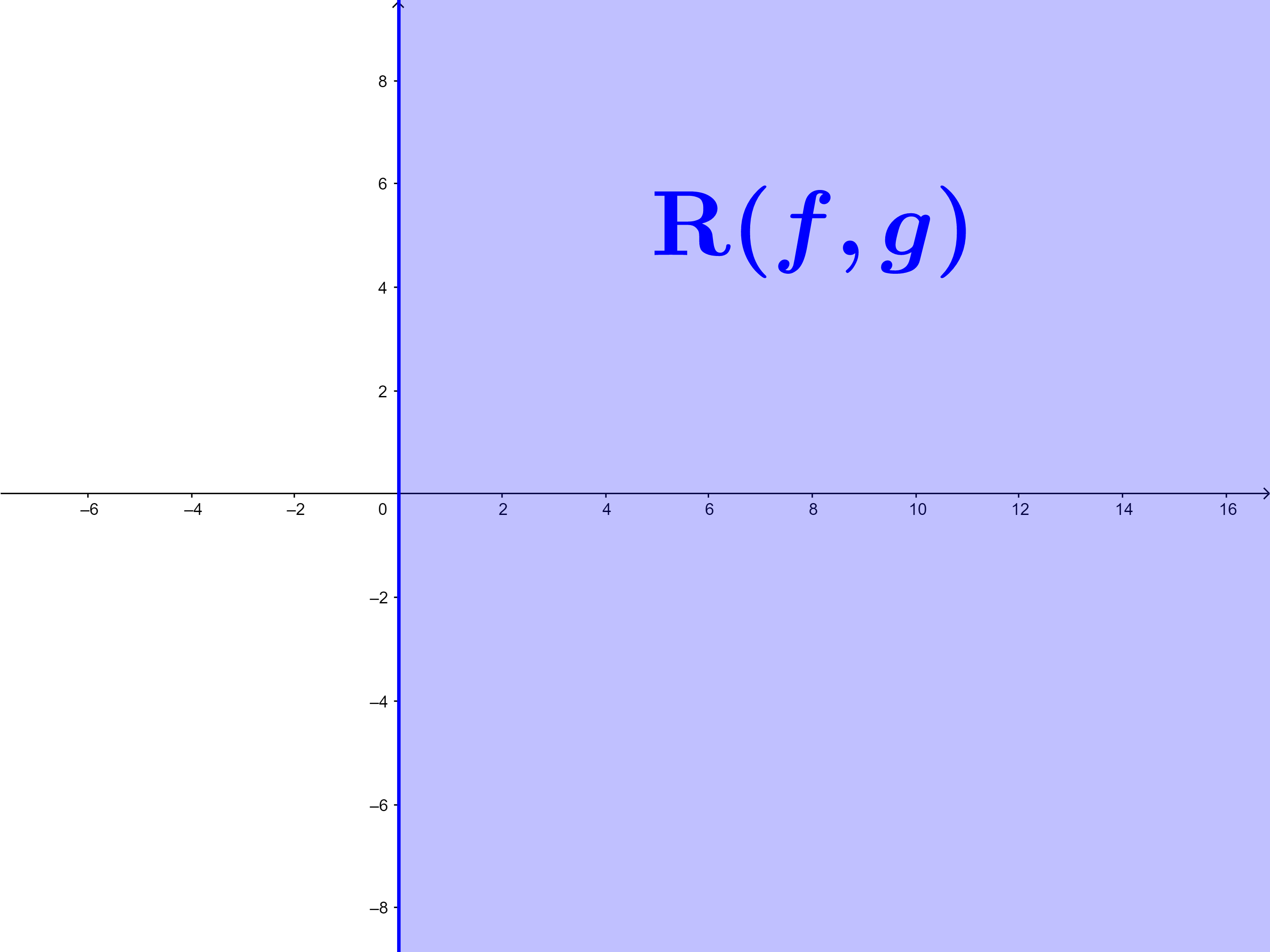}
	\caption{The joint range $\mathbf{R}(f,g)$ is the right half-plane.}
	\label{fig:shifted_nonhomo}
\end{subfigure}
\hfill
\begin{subfigure}[c]{0.49\textwidth}
	\centering
	\includegraphics[width=\linewidth]{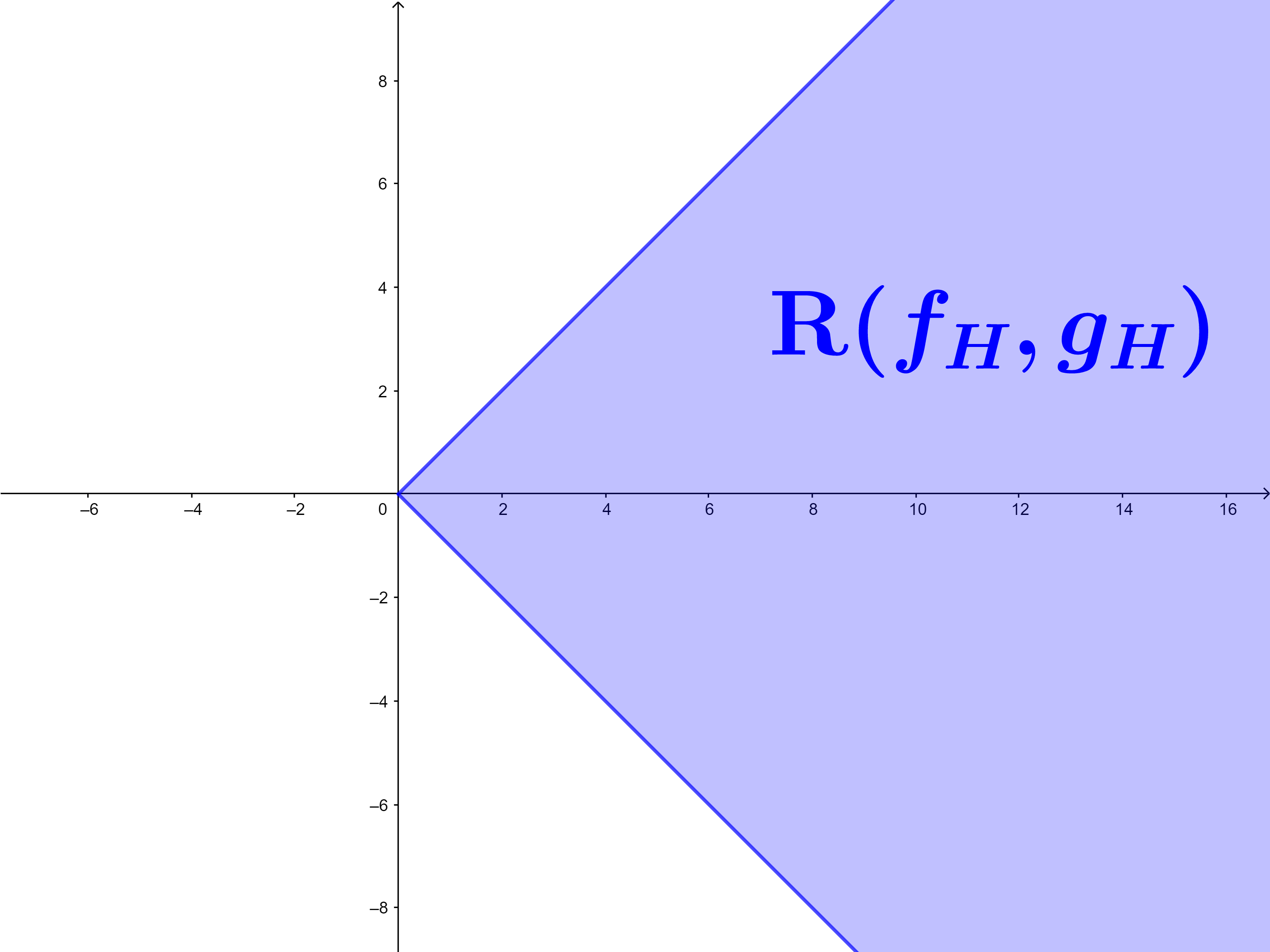}
	\caption{The joint range of the homogeneous part $\mathbf{R}(f_H,g_H)$ is an angular sector of angle $\pi/2$.}
	\label{fig:shifted_homo}
\end{subfigure}
\caption{Let $f(x,y,z) = x^2+y^2$ and $g(x,y,z) = -x^2+y^2+z$. }
\label{fig:shifted}
\end{figure}

Characterizing necessary and sufficient conditions for (P) from (\ref{eq:Homo+NonHomo-relation}) is technical and tedious.
Flores-Bazán and Opazo \cite[2016]{Bazan-Opazo16} wrote a 34-pages long paper to complete it. Moreover, they do not describe how more difficult the convexity can be tackled efficiently, i.e., how does one determine whether there exists an appropriate $d=(d_1,d_2) \in \mathbb{R}^2$, $d \neq 0$ such that the conditions {\bf (C1)-(C4)} (last row in Table \ref{tab:results_summary}) are satisfied?

The main purpose of this paper is to provide a different view and thus a different set of necessary and sufficient conditions to describe problem (P) with fully geometric insights. We indeed borrow the tools from Nguyen and Sheu \cite[2019]{Quang-Sheu19} in which a new concept called ``separation of quadratic level sets'' was introduced. The idea was first used to give a neat proof for the $\mathcal{S}$-lemma with equality by Xia et al. \cite[2016]{Xia-Wang-Sheu16}. We show in this paper that the same idea can be extended to accommodate our purpose to check the convexity of $\mathbf{R}(f,g)$ by a constructive polynomial-time procedure.

The paper is developed in the following sequence.
\begin{itemize}
  \item In Section \ref{sec:def_sep}, we shall review the concept of separation of two sets and list several properties for separation of two quadratic level sets.
  \item In Section \ref{sec:sep_JNR}, we obtain the main result ``the joint numerical range $\mathbf{R}(f,g)$ is non-convex if and only if there exists $\alpha, \beta \in \mathbb{R}$ such that $\{x\in\mathbb{R}^n |~g(x)=\beta\}$ separates $\{x\in\mathbb{R}^n |~f(x)=\alpha\}$ or $\{x\in\mathbb{R}^n |~f(x)=\alpha\}$ separates $\{x\in\mathbb{R}^n |~g(x)=\beta\}$.
  \item In Section \ref{sec:algorithm}, a polynomial-time procedure for checking the convexity of $\mathbf{R}(f,g)$ is provided.
  \item In Section \ref{sec:equiv}, we prove that the result of Flores-Bazán and Opazo \cite[2016]{Bazan-Opazo16} is an implication of our result in Section \ref{sec:sep_JNR}.
  \item Finally, we conclude our paper by briefly mentioning the relation between the convexity of joint range $\mathbf{R}(f,g)$ and variants of $\mathcal{S}$-lemma.
\end{itemize}

Throughout the paper, we adopt the following notations. For a matrix $P$, the symbol $P^{\dagger}$ denotes the Moore-Penrose generalized inverse of $P$. The null space and range space of $P$ is denoted by $\mathcal{N}(P)$ and $\mathcal{R}(P)$ respectively. For a subspace $W$  of  $\mathbb{R}^n$, $W^{\perp}$ denotes the orthogonal complement of $W$ with respect to the standard inner product equipped on $\mathbb{R}^n$. For conciseness, given a real constant $\alpha$, the set $\{f=\alpha\}\triangleq \{x\in\mathbb{R}^n |~f(x)=\alpha\}$ is called the $\alpha$-level set of $f$, and $\{f<\alpha\}\triangleq\{x\in\mathbb{R}^n |~f(x)<\alpha\}$ is said to be the $\alpha$-sublevel set of $f$.
\section{Separation of Quadratic Level Sets} \label{sec:def_sep}
Given a pair of quadratic functions $f$ and $g,$ Nguyen and Sheu in \cite[2019]{Quang-Sheu19} introduced the following definition for ``separation of level sets'':

\begin{definition}[\cite{Quang-Sheu19}] \label{def-separation}
The $0$-level set $\{g=0\}$ is said to separate the set $\{f\star 0\}$, where $\star \in \{<,=\}$, if there are non-empty subsets $L^-$ and $L^+$ of $\{f \star 0\}$ such that
\begin{equation} \label{eq:def_sep}
\begin{aligned}
&L^-\cup L^+ = \{f \star 0\} ~ \text{ and } \\
&g(u^-)g(u^+)<0,~\forall~ u^-\in L^-;~\forall~u^+\in L^+.
\end{aligned}
\end{equation}
\end{definition}

Several remarks directly from the definition should be noted.

\begin{enumerate}[(a)]
  \item  When $\{g = 0\}$ separates $\{f \star 0\}$, $\{f\star 0\}$ must be disconnected. Otherwise, by the Intermediate Value Theorem, $g(\{f\star 0\})$ is a connected interval containing the point $0,$ which contradicts to Definition \ref{def-separation}. \label{defRemark-affine}
  \item An affine set of the type $\{f \star 0\}$ with $\star \in \{<,=\}$ cannot be separated by any other $\{g=0\}.$ This is due to $\{f \star 0\}$ for $\star \in \{<,=\}$ being always connected, provided that $f$ is affine.
  \item $\{g = 0\}$ separates $\{f <0\}$ $\centernot\Longrightarrow$ $\{g = 0\}$ separates $\{f = 0\}.$ Figure \ref{fig:sep-intro-sublevel-not-level} shows such an example. \label{defRemark-sublevel-not-level}
  \item $\{g = 0\}$ separates $\{f =0\}$ $\centernot\Longrightarrow$ $\{g = 0\}$ separates $\{f < 0\}.$ See Figure \ref{fig:sep-intro-level-not-sublevel}. \label{defRemark-level-not-sublevel}
  \item $\{g = 0\}$ separates $\{f = 0\}$ $\centernot\Longrightarrow$ $\{f = 0\}$ separates $\{g = 0\}.$ A simple example is that $g$ is affine.
      A hyperplane $\{g = 0\}$ can easily separate the other supersurface $\{f = 0\},$ but it may not be separated reversely. See Figure \ref{fig:sep-intro-sublevel-not-level}.\label{defRemark-symmetry}
  \item $\{g = 0\}$ separates $\{f =0\}$ $\centernot\Longrightarrow$ $\{g = \alpha\}$ separates $\{f = \beta\}$ for $\alpha,~\beta\in\mathbb{R}.$ That is, for separation property, the ``hight'' associated with the contour matters. See Figure \ref{fig:sep-intro-diff-level}. \label{defRemark-height}
\end{enumerate}

%

\begin{figure}
\centering
\begin{minipage}[t]{0.485\linewidth}
\centering
\captionsetup{width=\linewidth}
\includegraphics[width=\linewidth]{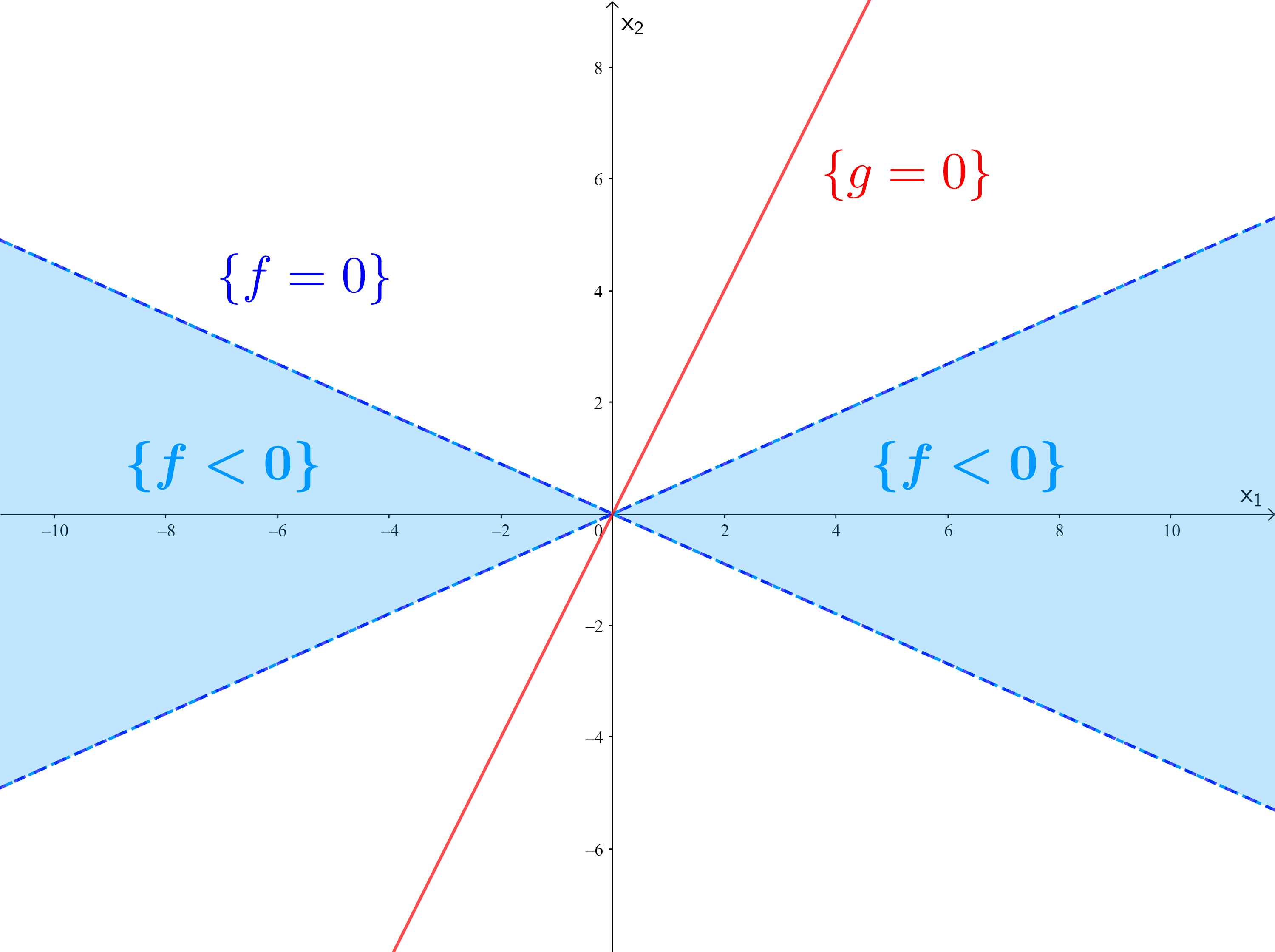}
\caption{\small For remark (\ref{defRemark-sublevel-not-level}) and remark (\ref{defRemark-symmetry}). Let $f(x,y) = -x^2 + 4 y^2$ and $g(x,y) = 2x-y$. The level set $\{g=0\}$ separates $\{f<0\},$ while $\{g=0\}$ does not separate $\{f=0\}.$}
\label{fig:sep-intro-sublevel-not-level}
\end{minipage}
\hfill
\begin{minipage}[t]{0.485\linewidth}
\centering
\captionsetup{width=\linewidth}
\includegraphics[width=\linewidth]{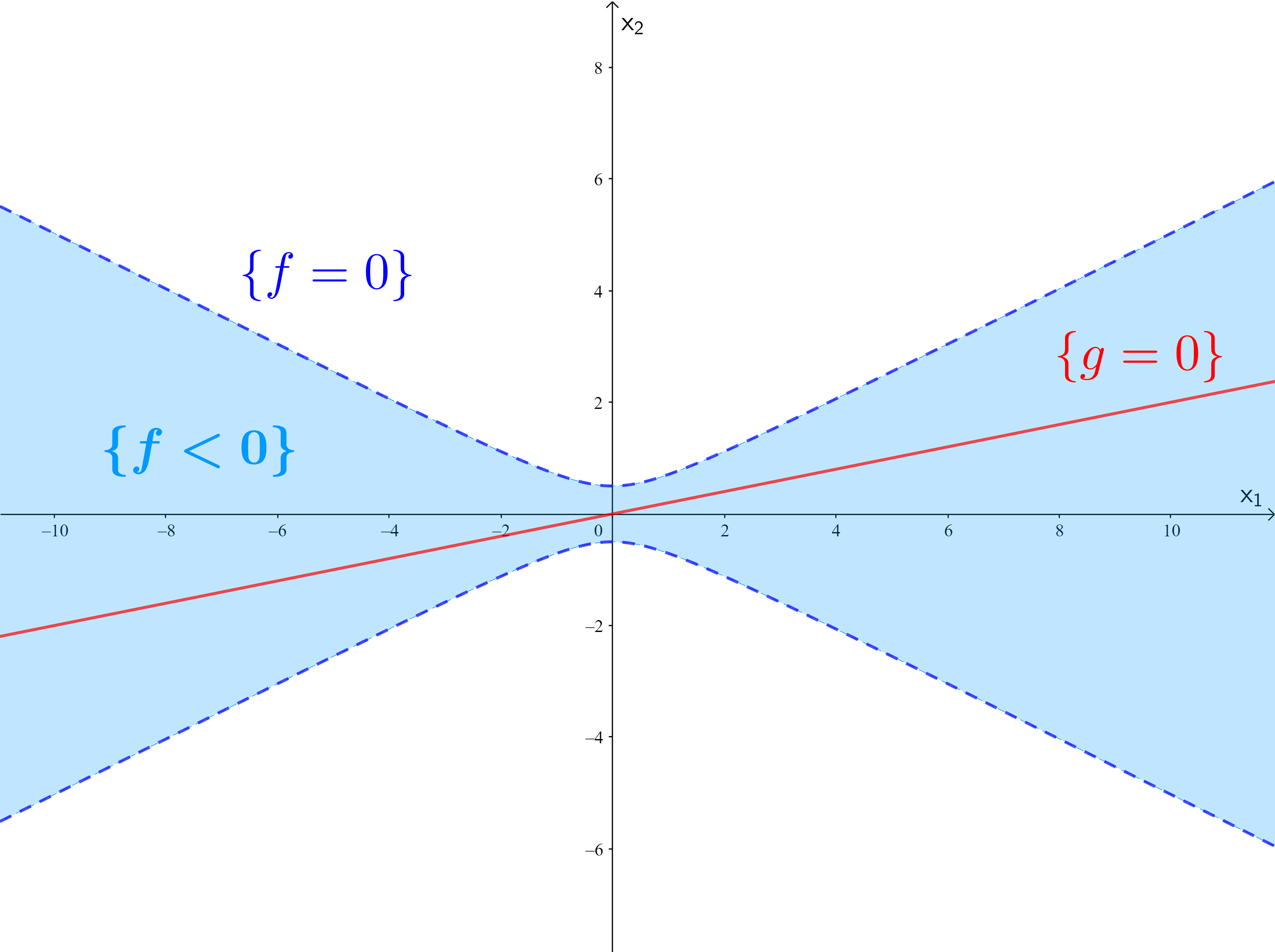}
\caption{\small For remark (\ref{defRemark-level-not-sublevel}). Let $f(x,y) = -x^2 + 4 y^2 - 1$ and $g(x,y) = x-5y$. The level set $\{g=0\}$ separates $\{f=0\}$ while $\{g=0\}$ does not separate $\{f<0\}.$}
\label{fig:sep-intro-level-not-sublevel}
\end{minipage}
\end{figure}
\begin{figure}
  \begin{subfigure}[t]{0.485\linewidth}
	\centering
	\captionsetup{width=\linewidth}
	\includegraphics[width=\linewidth]{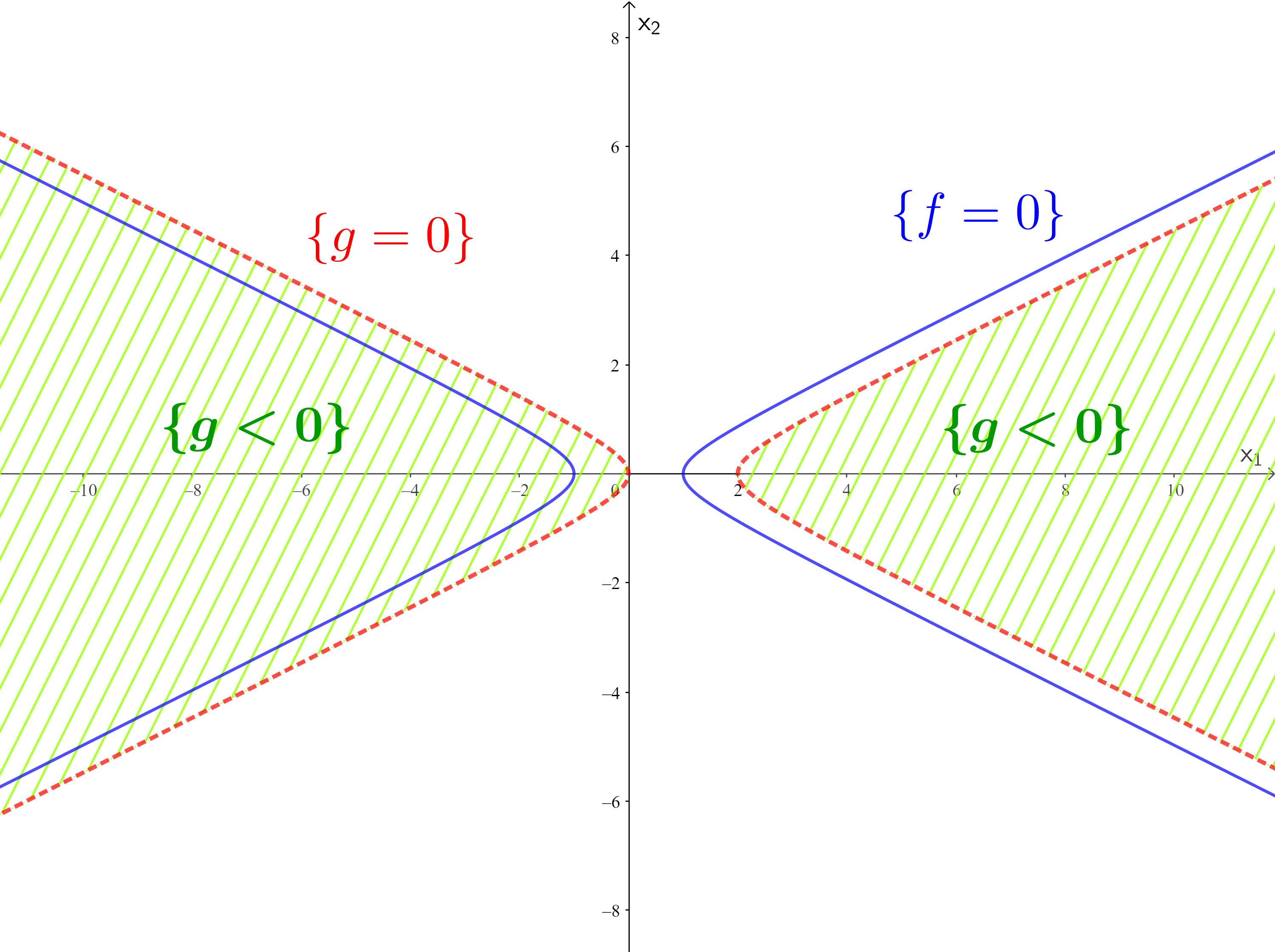}
	\caption{The level set $\{g=0\}$ separates $\{f=0\}$ since one branch of $\{f=0\}$ lies in $\{g<0\}$ while the other lies in $\{g>0\}$.}
	\label{fig:sep-intro-diff-level-1}
\end{subfigure}
\hfill
\begin{subfigure}[t]{0.485\linewidth}
	\centering
	\captionsetup{width=\linewidth}
	\includegraphics[width=\linewidth]{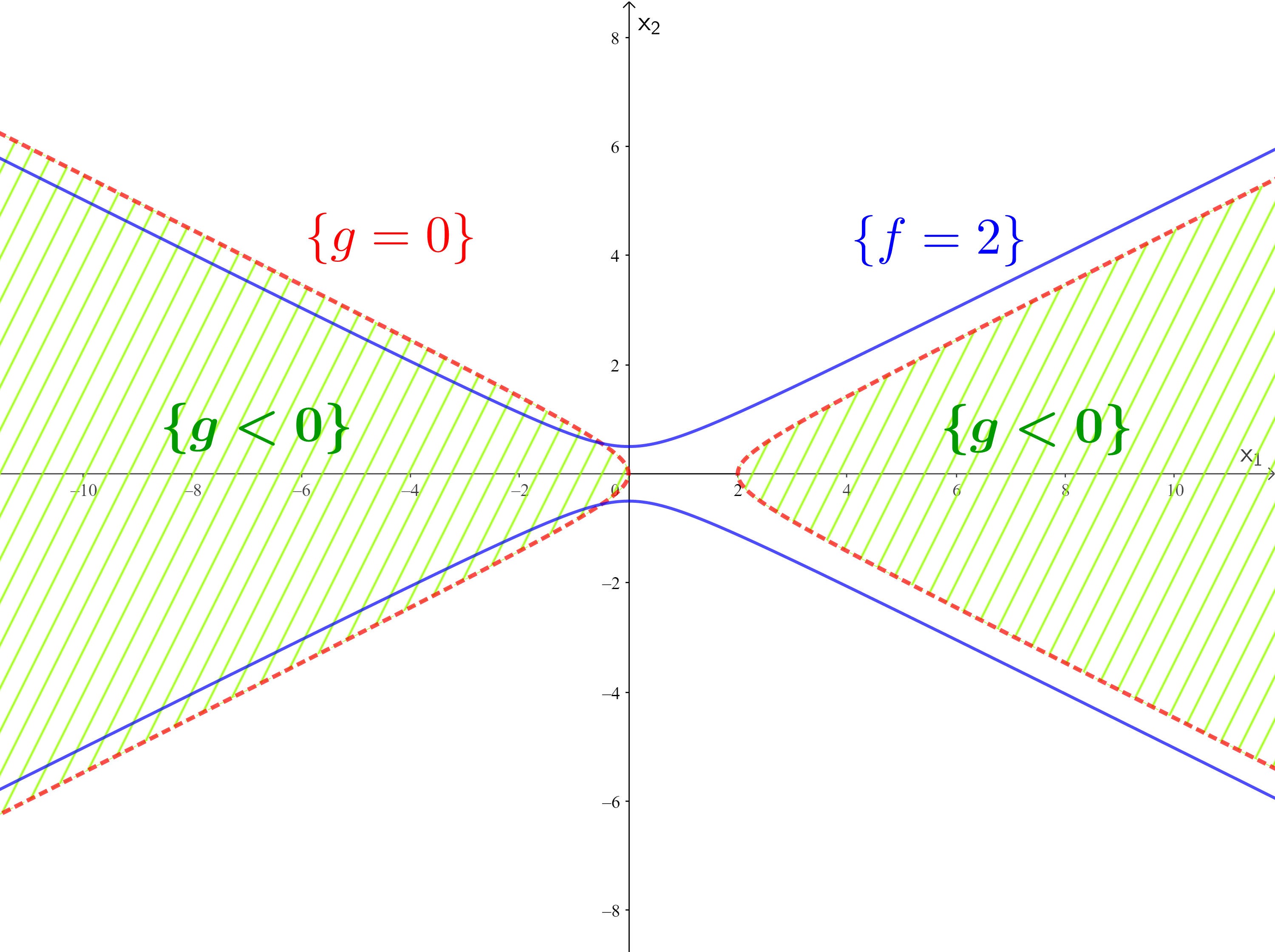}
	\caption{Since $\{g=0\} \cap \{f=2\}\not=\emptyset,$ $\{g=0\}$ does not separate $\{f=2\}$ and $\{f=2\}$ does not separate $\{g=0\}$ either.}
	\label{fig:sep-intro-diff-level-2}
\end{subfigure}
\caption{For remark (\ref{defRemark-height}) in which $f(x,y) = -x^2 + 4 y^2 + 1$ and $g(x,y) = -(x-1)^2+4y^2+1$.}
\label{fig:sep-intro-diff-level}
\end{figure}

In the above remarks, four types of separation were considered.
\begin{itemize}
  \item[-] $\{g = 0\}$ separates $\{f <0\}$;
  \item[-] $\{g = 0\}$ separates $\{f =0\}$;
  \item[-] $\{g = 0\}$ separates $\{f =0\}$, and simultaneously $\{f =0\}$ separates $\{g =0\}.$
  \item[-] $\{g = \alpha\}$ separates $\{f =\beta\}$ for some $\alpha,\beta\in\mathbb{R}.$
\end{itemize}
For $\{g = 0\}$ to separate $\{f <0\},$ it has been completely characterized by
Nguyen and Sheu \cite[2019]{Quang-Sheu19}. They showed that $\{g = 0\}$ can separate $\{f <0\}$ if and only if $g$ is affine, the matrix $A$ (of $f$) has exactly one negative eigenvalue, and $f \left|_{\{g=0\}}(x) \right. \geq 0$, where $f \left|_{\{g=0\}}(x) \right.$ is the restriction of the function $f$ on the set $\{g=0\}$.

The case for $\{g = 0\}$ to separate $\{f=0\}$ is quite different. While $\{g = 0\}$ to separate $\{f <0\}$ can happen only when $g$ is affine, Figure \ref{fig:sep-intro-diff-level-1} demonstrates the possibility for a quadratic level set $\{g = 0\}$ to separate another quadratic level set $\{f=0\}.$ Interestingly, in an unpublished manuscript by Nguyen and Sheu \cite[2020]{separation}, it was shown that, for $\{g = 0\}$ to separate  $\{f=0\},$ there must exist a linear combination $-\lambda f+g$ such that $-\lambda f+g$ is affine and $\{-\lambda f+g=0\}$ separates $\{f=0\}.$

\begin{lemma}[Nguyen and Sheu {\cite[2020]{separation}}] \label{lem:G=kF+H}
The 0-level set $\{g=0\}$ separates $\{f=0\}$ if and only if there exists some $\lambda \in \mathbb{R}$ such that $B = \lambda A$ and $\{-\lambda f + g = 0\}$ separates $\{f=0\}$.
\end{lemma}

Lemma \ref{lem:G=kF+H} limits $\{g=0\}$ to separate $\{f=0\}$ for only two cases.
Due to $\{f=0\}$ being separated, it must be disconnected so that it is of a hyperbolic type (under a suitable basis) adopting either the following form
$$f(x)=-x_1^2+\delta (x_{2}^2+\cdots + x_m^2)+1,~\delta\in\{0,1\};$$
or
$$f(x)=x_1^2-\delta (x_{2}^2+\cdots + x_m^2)-1,~\delta\in\{0,1\}.$$
See Nguyen and Sheu \cite[2020]{separation}. By Lemma \ref{lem:G=kF+H}, when $\{g=0\}$ separates $\{f=0\}$, we know $B = \lambda A$ so that $g$ is either affine (in case $\lambda=0$), or is also of a hyperbolic type like $f$ (in case $\lambda\not=0$). The former one results in an affine level set $\{g=0\}$ to separate a hyperbolic level set $\{f=0\}$ as illustrated by Figure \ref{fig:sep-intro-level-not-sublevel}. The letter is a case that a hyperbolic level set $\{g=0\}$ separates another hyperbolic level set $\{f=0\}.$ See Figure \ref{fig:sep-intro-diff-level-1} for an example. In either case, Lemma \ref{lem:G=kF+H} reduces the separation of $\{f=0\}$ by $\{g=0\}$ to the separation of $\{f=0\}$ by a hyperplane $\{-\lambda f + g = 0\}$ with a proper parameter $\lambda.$ The following lemma characterizes the necessary and sufficient conditions for an affine level set to separate a quadratic level set.

\begin{lemma}[Nguyen and Sheu {\cite[2020]{separation}}] \label{lem:chckable}
Suppose that $h(x) = c^T x + c_0$ is affine. Then $\{h=0\}$ separates $\{f=0\}$ if and only if either $\left( \bar{f}, \bar{A}, \bar{a} \right) = \left( f, A, a \right)$ or $\left( \bar{f}, \bar{A}, \bar{a} \right) = \left( -f, -A, -a \right)$ satisfies the following three conditions:
\begin{enumerate}[(i)]
\setlength{\itemsep}{4pt}
\item $\bar{A}$ has exactly one negative eigenvalue, $\bar{a} \in \mathcal{R}(\bar{A})$
\item $c \in \mathcal{R}(\bar{A})$, $c \neq 0$
\item $V^T \bar{A} V \succeq 0$, $\bar{w} \in \mathcal{R}(V^T \bar{A} V)$, and $\bar{f}(x_0)-\bar{w}^T (V^T \bar{A} V^T)^{\dagger} \bar{w} > 0$,
\end{enumerate}
where $\bar{w} = V^T(\bar{A}x_0+\bar{a})$, $x_0 = \frac{-c_0}{c^T c} c$, and $V \in \mathbb{R}^{n \times (n-1)}$ is the matrix basis for $\mathcal{N}(c^T)$.
\end{lemma}

Furthermore, when it happens that a hyperbolic level set $\{g=0\}$ separates another hyperbolic level set $\{f=0\},$ as Figure \ref{fig:sep-intro-diff-level-1} suggests, the reverse that $\{f=0\}$ separates $\{g=0\}$ also holds true.

\begin{lemma}[Nguyen and Sheu {\cite[2020]{separation}}] \label{lem:mutual-sep}
Suppose that both $f$ and $g$ are quadratic functions. If $\{g=0\}$ separates $\{f=0\}$, then $\{f=0\}$ separates $\{g=0\}$ also.
\end{lemma}

In the case of Lemma \ref{lem:mutual-sep}, we say that {\it $\{g=0\}$ and $\{f=0\}$ mutually separate.}

The following proposition shows that the separation property can be preserved under linear combinations.

\begin{proposition}\label{prop:linear_comb_preserve_separation}
If $\{g=0\}$ separates $\{f=0\}$, then $\{\eta f+\theta g=0\}$ separates $\{\sigma f=0\}$ for all $\eta, \theta, \sigma \in \mathbb{R}$ with $\theta \neq 0$, $\sigma \neq 0$.
\end{proposition}
\begin{proof}
Since $\{g=0\}$ separates $\{f=0\}$, there are non-empty subsets $L^+$ and $L^-$ of $ \{f = 0\}$ such that
\begin{equation*}
\begin{aligned}
&L^-\cup L^+ = \{f = 0\} ~ \text{ and } \\
&g(u^-)g(u^+)<0,~\forall~ u^-\in L^-;~\forall~u^+\in L^+.
\end{aligned}
\end{equation*}
By $\{f = 0\} = \{ \sigma f = 0 \}$ for any $\sigma \neq 0,$ we have $L^-\cup L^+ = \{\sigma f = 0\}$. Moreover, for any $\theta \neq 0$,
$$(\eta f + \theta g)(u)(\eta f + \theta g)(v) = \theta^2 g(u)g(v)<0 ~ \text{ for all } u \in L^+, v\in L^-,$$
which shows that $\{\eta f+\theta g=0\}$ separates $\{\sigma f=0\}$ for all $\theta \neq 0$, $\sigma \neq 0$.
\end{proof}

The following proposition can be viewed as the converse of Proposition \ref{prop:linear_comb_preserve_separation}.

\begin{proposition} \label{prop:two_sides_linear_comb}
Suppose that $\{\eta f  + \theta g = 0\}$ separates $\{\sigma f + \tau g = 0\}$ for some real numbers $\eta, \theta, \sigma, \tau$ with $\sigma \theta -\tau \eta \neq 0$. Then,
\begin{enumerate}[(a)]
\item if both $f$ and $g$ are quadratic functions, the 0-level sets $\{f=0\}$ and $\{g=0\}$ mutually separate each other;
\item if one of $f$ and $g$ is an affine function, then the other must be a quadratic function and the affine 0-level set separates the quadratic 0-level set.
\end{enumerate}
Namely, if $\{\eta f  + \theta g = 0\}$ separates $\{\sigma f + \tau g = 0\}$ with $\sigma \theta -\tau \eta \neq 0$, then $\{g=0\}$ separates $\{f=0\}$ or $\{f=0\}$ separates $\{g=0\}$.
\end{proposition}
\begin{proof}
Since $\{\eta f  + \theta g = 0\}$ separates $\{\sigma f + \tau g = 0\}$, $\sigma f + \tau g$ cannot be affine function so that one of $f$ and $g$ is quadratic. Let us assume that $f$ is quadratic.

\noindent{-} If $\tau \neq 0$: by Proposition \ref{prop:linear_comb_preserve_separation}, we have
$$\left\{\frac{\theta}{\sigma \theta -\tau \eta} (\sigma f + \tau g)  -\frac{\tau}{\sigma \theta -\tau \eta}(\eta f + \theta g)=0\right\}\hbox{ separates }\left\{\frac{1}{\tau}(\sigma f + \tau g)=0\right\}. $$
After simplifying,
$$\{f = 0\} \text{ separates } \left\{ \frac{\sigma}{\tau} f + g = 0 \right\}.$$
Since both $f$ and $\frac{\sigma}{\tau} f + g$ are quadratic, Lemma \ref{lem:mutual-sep} ensures mutual separation so that $\{ \frac{\sigma}{\tau} f + g = 0\}$ separates $\{f = 0\}$, too. Applying Proposition \ref{prop:linear_comb_preserve_separation} again, we have $\{g=0\}$ separates $\{f=0\}$.

\noindent{-} If $\tau = 0$: our assumption reduces to $\{\eta f  + \theta g = 0\}$ separates $\{\sigma f = 0\}$. Since $\sigma \theta -\tau \eta \neq 0$, we have $\sigma \neq 0$ and $\theta \neq 0$. By Proposition \ref{prop:linear_comb_preserve_separation},
$$\left\{-\frac{\eta}{\sigma \theta} (\sigma f) + \frac{1}{\theta}(\eta f + \theta g)=0\right\}\hbox{ separates }\left\{\frac{1}{\sigma}(\sigma f)=0\right\}. $$
It yields $\{g = 0\} \text{ separates } \{f = 0\}.$

Therefore, if a linear combination of $f$ and $g$ separates another linearly independent combination of them, and if $f$ is quadratic, we have $\{g=0\}$ separates $\{f=0\}$. Similarly, when $g$ is quadratic, we obtain $\{f=0\}$ separates $\{g=0\}$. In summary, under the assumption, we have
\begin{align}
f \text{ is quadratic } &\Longrightarrow  \{g=0\} \text{ separates } \{f=0\} \label{eq:pf-prop-f-quadratic} \\
g \text{ is quadratic } &\Longrightarrow  \{f=0\} \text{ separates } \{g=0\}. \label{eq:pf-prop-g-quadratic}
\end{align}
Since one of $f$ and $g$ must be quadratic, one of (\ref{eq:pf-prop-f-quadratic}) and (\ref{eq:pf-prop-g-quadratic}) must happen. Hence, 
\begin{itemize}
\item[-] when both $f$ and $g$ are quadratic, $\{f=0\}$ and $\{g=0\}$ mutually separate;
\item[-] when one of $f$ and $g$ is affine, the affine 0-level set separates the quadratic 0-level set.
\end{itemize}
\end{proof}

\begin{proposition}\label{prop:criterion}
The 0-level set $\{g=0\}$ separates the 0-level set $\{f=0\}$ if and only if
\begin{align}
&\{f=0\}\cap\{g=0\}=\emptyset ~ \text{ and } \label{eq:alt_def-1} \\
&g(u)g(v)<0 ~ \text{ for some }~u, v \in\{f=0\}. \label{eq:alt_def-2}
\end{align}
\end{proposition}
\begin{proof} The necessity follows directly from Definition \ref{def-separation}. It suffices to prove the sufficiency. From
\eqref{eq:alt_def-1}, we know $\{f=0\}\subset \{g<0\}\cup\{g>0\}.$ Let
\begin{equation} \label{eq:pf-L+-}
L^- = \{f=0\} \cap \{g<0\} ~ \text{ and } ~ L^+ = \{f=0\} \cap \{g>0\}.
\end{equation}
Due to (\ref{eq:alt_def-2}), $L^-\not=\emptyset$ and $L^+\not=\emptyset$.
Therefore, $\{f=0\} = L^+ \cup L^-$. Finally, (\ref{eq:pf-L+-}) directly implies that $g(u^+)g(u^-)<0$ for all $u^+ \in L^+$ and $u^- \in L^-.$ The proof is  complete.
\end{proof}
\section{Characterising Non-convexity of $\mathbf{R}(f,g)$ by Separation of Level Sets} \label{sec:sep_JNR}
Now we are ready to show that the non-convexity of the joint range $\mathbf{R}(f,g)$ for two quadratic mappings, homogeneous or not, can be completely characterized by the separation feature of their level sets.
This constitutes the main theme of the section.

\begin{theorem} \label{thm:noncvx_sep}
The joint numerical range $\mathbf{R}(f,g)$ is non-convex if and only if there exists $\alpha, \beta \in \mathbb{R}$ such that $\{g=\beta\}$ separates $\{f=\alpha\}$ or $\{f=\alpha\}$ separates $\{g=\beta\}$. More precisely,
\begin{enumerate}[(a)]
\item when both $f$ and $g$ are quadratic functions, $\mathbf{R}(f,g)$ is non-convex if and only if $(\exists \alpha, \beta \in \mathbb{R})~\{f=\alpha\}$ and $\{g=\beta\}$ mutually separates each other;
\item when one of $f$ and $g$ is affine, $\mathbf{R}(f,g)$ is non-convex if and only if $(\exists \alpha, \beta \in \mathbb{R})$ the affine $\beta$-level set separates the quadratic $\alpha$-level set.
\end{enumerate}
\end{theorem}
\begin{proof}
For convenience, let us adopt
$$\mathbf{R}(f,g) =
\left\{ (t,k) \in \mathbb{R}^2
\left| ~
\begin{array}{r}
t = f(x)  \\
k = g(x)
\end{array} ~ , ~ x \in \mathbb{R}^n
\right. \right\}.$$
[Proof for necessity]: The set $\mathbf{R}(f,g)$ is non-convex if and only if there exists a triple of points $(M, N, K)$ satisfying
\begin{align}
&M, N \in \mathbf{R}(f,g), \label{eq:noncvx_segment-1} \\
&M, N, K \text{ are colinear}, \label{eq:noncvx_segment-2} \\
&K \text{ lies between the segment } \overline{MN}, \label{eq:noncvx_segment-3} \\
&K \notin \mathbf{R}(f,g). \label{eq:noncvx_segment-4}
\end{align}

\begin{figure}
\centering
\includegraphics[width=0.6\linewidth]{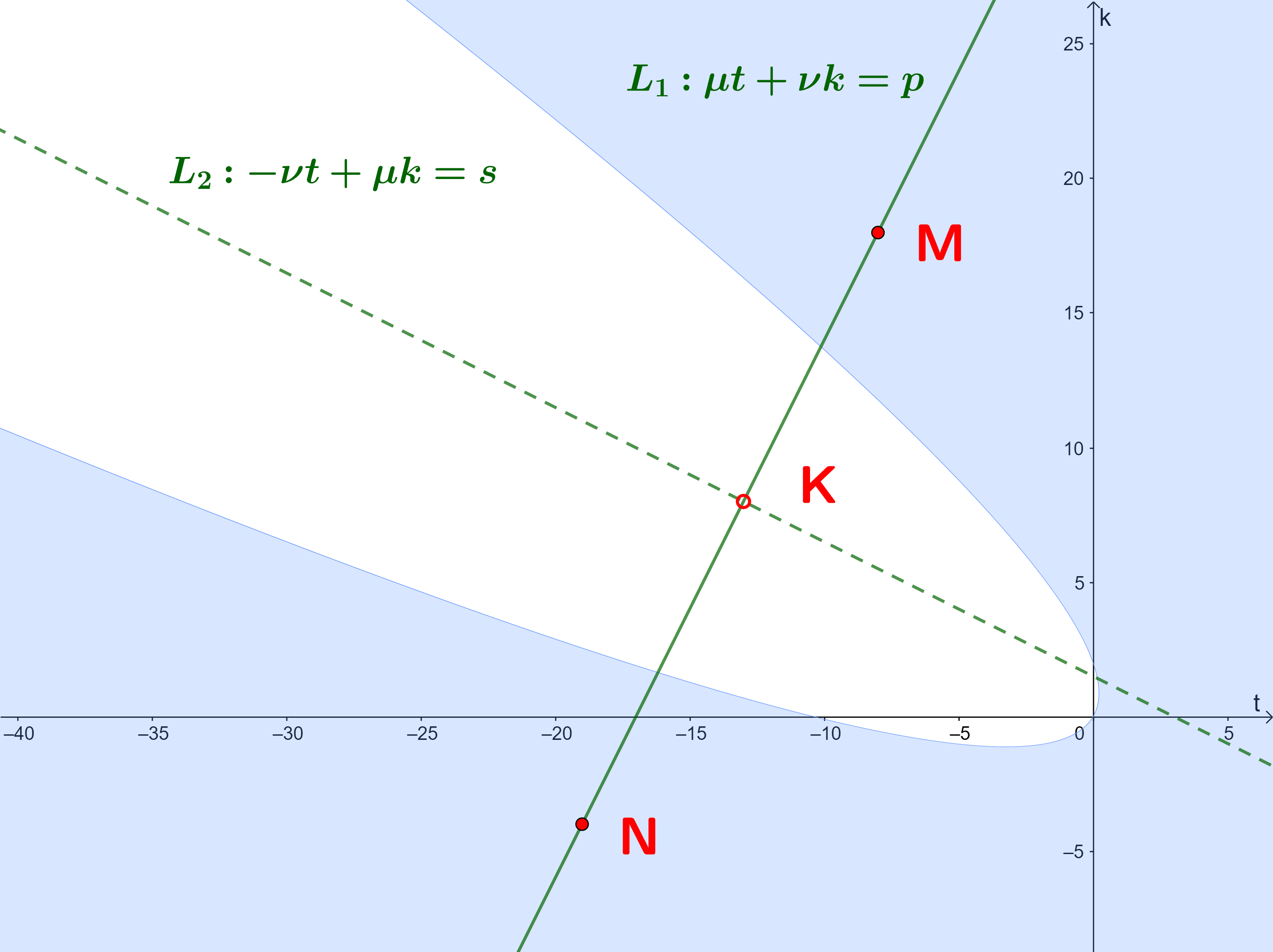}
\caption{Graph for Proof of Theorem \ref{thm:noncvx_sep}.}
\label{fig:for-proof}
\end{figure}

Let $L_1 :  \mu t + \nu k = p,~\mu^2+\nu^2 \neq 0$ be the line passing through $M$, $N$, and $K$. Since both $M$ and $N$ lie in $L_1 \cap \mathbf{R}(f,g)$, there exists some $u^{\ast}, v^{\ast}$ such that
$M = \left( f(u^{\ast}) , g(u^{\ast}) \right),~N = \left( f(v^{\ast}), g(v^{\ast}) \right)$ and
\begin{equation}\label{eq:pf_alt-def-2.*}
u^{\ast}, v^{\ast} \in \left\{ x \in \mathbb{R}^n |~ \mu f(x) + \nu g(x) = p \right\}.
\end{equation}

Moreover, let $L_2 : -\nu t + \mu k = s$ be perpendicular to $L_1$ with the intersection $K$.  Since the two points $M$ and $N$ lie in different sides of $L_2$,
\begin{equation} \label{eq:pf_alt-def-2}
\left( -\nu f(u^{\ast}) + \mu g(u^{\ast}) - s \right)\left( -\nu f(v^{\ast}) + \mu g(v^{\ast}) - s \right) < 0.
\end{equation}
In addition, $K\not\in\mathbf{R}(f,g),$ we have
\begin{equation} \label{eq:pf_alt-def-1}
\left\{ x \in \mathbb{R}^n |~ \mu f(x) + \nu g(x) - p = 0 \right\} \cap \left\{ x \in \mathbb{R}^n | -\nu f(x) + \mu g(x) - s = 0 \right\} = \emptyset.
\end{equation}

Combining (\ref{eq:pf_alt-def-2.*}), (\ref{eq:pf_alt-def-2}), and (\ref{eq:pf_alt-def-1}) together, we apply Proposition \ref{prop:criterion} to conclude that
\begin{equation}\label{dfg-1}
\{\mu f + \nu g = p\}~\hbox{ separates }~\{-\nu f + \mu g = s\}.
\end{equation}
Since $\mu^2+\nu^2 \neq 0$, \eqref{dfg-1} can be recast as
$$
\left\{ \mu \left( f-\alpha \right) + \nu (g-\beta) = 0 \right\} \text{ separates } \left\{-\nu (f-\alpha) + \mu \left(g-\beta \right) = 0 \right\},
$$
where $\alpha = \frac{p \mu - s \nu}{\mu^2+\nu^2}$ and $\beta = \frac{p \nu + s \mu}{\mu^2+\nu^2}$. The nessesity part of the theorem follows immediately from
Proposition \ref{prop:two_sides_linear_comb}.

\vskip0.2cm
\noindent [Proof for sufficiency]: Suppose that $\{g=\beta\}$ separates $\{f=\alpha\}$ for some $\alpha, \beta \in \mathbb{R}$. Proposition \ref{prop:criterion} implies that
\begin{align}
&\{g=\beta\}\cap\{f=\alpha\}=\emptyset \label{eq:ex_general_sep_cond1}\\
&(g(u^{\ast})-\beta)(g(v^{\ast})-\beta)<0 \text{ for some } u^{\ast}, v^{\ast} \in \{f=\alpha\}. \label{eq:ex_general_sep_cond2}
\end{align}
Set the points $M, N, K$ as
\begin{align*}
M &= \left( f(u^{\ast}) , g(u^{\ast}) \right) = \left( \alpha , g(u^{\ast}) \right), ~ \\
N &= \left( f(v^{\ast}) , g(v^{\ast}) \right) = \left( \alpha , g(v^{\ast}) \right), ~ \\
K &= \left( \alpha , \beta \right).
\end{align*}
Then, $M, N, K$ lie on the same vertical line $t=\alpha$. Moreover, $M, N \in \mathbf{R}(f,g)$ and, by (\ref{eq:ex_general_sep_cond2}), $K$ lies between $M$ and $N.$
However, by (\ref{eq:ex_general_sep_cond1}), it is impossible to have $x\in\mathbb{R}^n$ such that $f(x)=\alpha$ and $g(x)=\beta$ so that $K \notin \mathbf{R}(f,g)$. Therefore, $\mathbf{R}(f,g)$ is non-convex. The other symmetry case that $\{f=\alpha\}$ separates $\{g=\beta\}$ for some $\alpha, \beta \in \mathbb{R}$ can be analogously proved. The proof of Theorem \ref{thm:noncvx_sep} is thus complete.
\end{proof}

\begin{example} \label{ex:main}
Let two quadratic functions $f,~g$ be
\begin{align*}
f(x,y) &= -\frac{\sqrt{3}}{2} x^2 + \frac{\sqrt{3}}{2} y^2 + x - \frac{1}{2} y \\
g(x,y) &= \frac{1}{2} x^2 - \frac{1}{2} y^2 + \sqrt{3} x - \frac{\sqrt{3}}{2} y.
\end{align*}
Figure \ref{fig:ex_main-level} shows that $\{g=2\}$ and $\{f=-4\}$ mutually separate, and Figure \ref{fig:ex_main-JNR} shows that the joint range $\mathbf{R}(f,g)$ is non-convex. These two figures justify our main theorem, Theorem \ref{thm:noncvx_sep}.

\begin{figure}
\centering
\begin{subfigure}[t]{0.485\linewidth}
	\centering
	\captionsetup{width=\linewidth}
	\includegraphics[width=\linewidth]{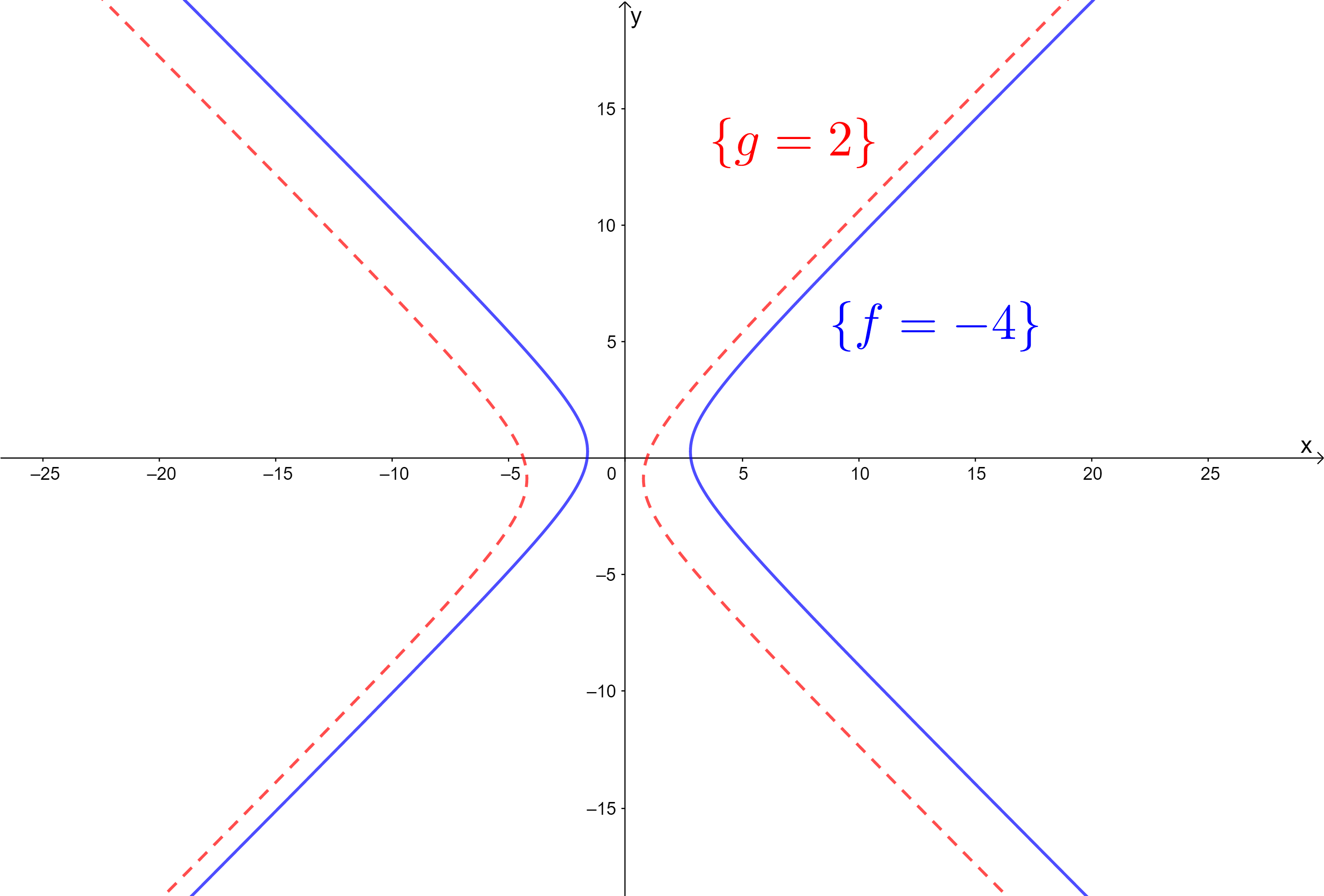}
	\caption{The level set $\{g=2\}$ and the level set $\{f=-4\}$ mutually separate.}
	\label{fig:ex_main-level}
\end{subfigure}
\hfill
\begin{subfigure}[t]{0.485\linewidth}
	\centering
	\captionsetup{width=\linewidth}
	\includegraphics[width=\linewidth]{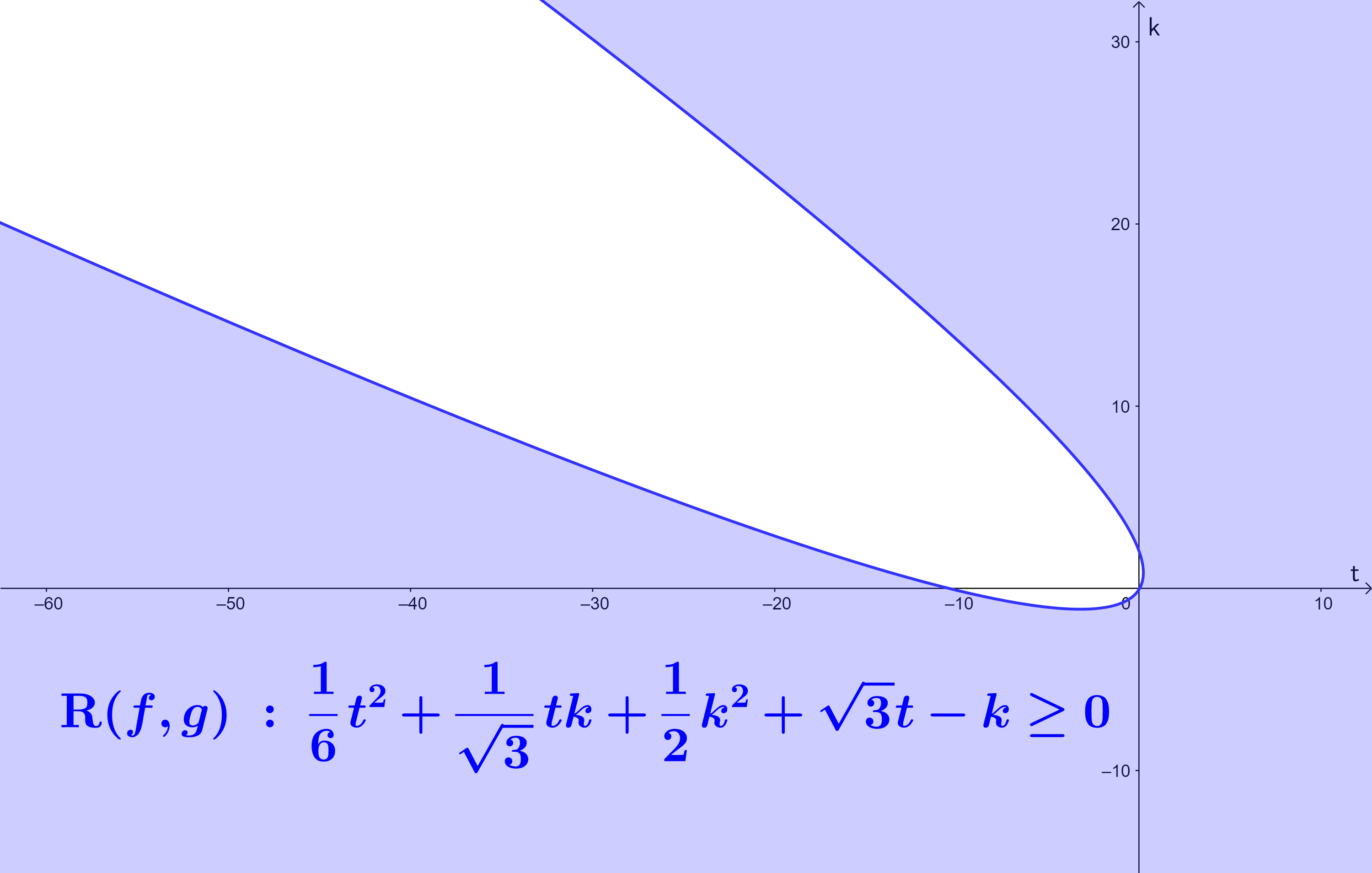}
	\caption{The joint range $\mathbf{R}(f,g)$ is non-convex while $\{g = \beta\}$ and $\{f=\alpha\}$ mutually separate.}
	\label{fig:ex_main-JNR}
\end{subfigure}
\caption{For Example \ref{ex:main}. Let $f(x,y) = -\frac{\sqrt{3}}{2} x^2 + \frac{\sqrt{3}}{2} y^2 + x - \frac{1}{2} y$ and $g(x,y) = \frac{1}{2} x^2 - \frac{1}{2} y^2 + \sqrt{3} x - \frac{\sqrt{3}}{2} y$.}
\label{fig:ex_main}
\end{figure}

\end{example}

\section{Algorithm for Checking Convexity of $\mathbf{R}(f,g)$} \label{sec:algorithm}
This section is devoted to a polynomial-time procedure for checking the convexity of $\mathbf{R}(f,g)$. If both $A$ and $B$ are zero, $\mathbf{R}(f,g)$ is the range of $\mathbb{R}^n$ under an affine transformation, which is convex. In the following, we assume that $A \neq 0.$ Then, Theorem \ref{thm:noncvx_sep} asserts that the joint range $\mathbf{R}(f,g)$ is non-convex if and only if there exists $\alpha, \beta \in \mathbb{R}$ such that 
\begin{equation} \label{eq:levels-sep}
\{g - \beta = 0\} \text{ separates } \{f - \alpha = 0\}.
\end{equation}
Lemma \ref{lem:G=kF+H} further ensures that (\ref{eq:levels-sep}) can be always reduced, by suitable linear combinations of $f$ and $g$, to the case that an affine level set separates a quadratic one. Specifically, (\ref{eq:levels-sep})
holds if and only if the quadratic matrices of $g-\beta$ and $f-\alpha$ are linearly dependent, namely $B = \lambda A$, and the affine level set $\{-\lambda (f - \alpha) + (g-\beta) = 0\}$ separates the quadratic level set $\{f - \alpha = 0\}$. Thus, we obtain the following corollary of Theorem \ref{thm:noncvx_sep}.

\begin{corollary} \label{cor:noncvx_linear_sep}
Suppose that $A \neq 0$. The joint numerical range $\mathbf{R}(f,g)$ is non-convex if and only if $B = \lambda A$ for some $\lambda \in \mathbb{R}$ and $\{-\lambda f + g = \gamma \}$ separates $\{f=\alpha\}$ for some $\alpha, \gamma \in \mathbb{R}$.
\end{corollary}

We observe that there are three constants in Corollary \ref{cor:noncvx_linear_sep}, $\lambda,~\gamma,~\alpha,$ the existence of which
are to be determined. The condition ``$B = \lambda A$ for some $\lambda \in \mathbb{R}$'' is easy to verify. If $B \not= \lambda A,~\forall\lambda \in \mathbb{R},$ $\mathbf{R}(f,g)$ is convex. Otherwise, $\mathbf{R}(f,g)$ can be non-convex only when there exists some $\gamma,~\alpha$ such that the hyperplane $\{-\lambda f + g = \gamma \}$ separates $\{f=\alpha\}$. By Lemma \ref{lem:chckable}, such a pair of $\gamma, \alpha$, if exist, must satisfy the following: either $\left( \bar{f}_{\alpha}, \bar{A}, \bar{a} \right) = \left( f - \alpha, A, a \right)$ or $\left( \bar{f}_{\alpha}, \bar{A}, \bar{a} \right) = \left( -\left( f - \alpha \right), -A, -a \right)$ satisfies
\begin{enumerate}[(i)]
\setlength{\itemsep}{4pt}
\item $\bar{A}$ has exactly one negative eigenvalue, $\bar{a} \in \mathcal{R}(\bar{A})$
\item $c = -\lambda a + b \in \mathcal{R}(\bar{A})$, $c \neq 0$
\item $V^T \bar{A} V \succeq 0$, $\bar{w}_{\gamma} \in \mathcal{R}(V^T \bar{A} V )$, and $\bar{f}_{\alpha}(x_{\gamma})- \bar{w}_{\gamma}^T ( V^T \bar{A} V^T )^{\dagger} \bar{w}_{\gamma} > 0$,
\end{enumerate}
where $\bar{w}_{\gamma} = V^T(\bar{A}x_{\gamma}+\bar{a})$, $x_{\gamma} = \frac{-(-\lambda a_0 + b_0 - \gamma)}{c^T c} c$, and $V \in \mathbb{R}^{n \times (n-1)}$ is the matrix basis for $\mathcal{N}(c^T)$.

However, among (i)-(iii), we find that $\alpha,~\gamma$ appear only in
\begin{align}
&\bar{w}_{\gamma} \in \mathcal{R}(V^T \bar{A} V),~\bar{w}_{\gamma} = V^T(\bar{A}x_{\gamma}+\bar{a}),~x_{\gamma} = \frac{-(-\lambda a_0 + b_0 - \gamma)}{c^T c} c \label{eq:constant-dependent-gamma} \\
&\bar{f}_{\alpha}(x_{\gamma}) - \bar{w}_{\gamma} ^T (V^T A V^T)^{\dagger} \bar{w}_{\gamma} > 0. \label{eq:constant-dependent-alpha-gamma}
\end{align}
where $\left( \bar{f}_{\alpha}, \bar{A}, \bar{a} \right) = \left( f - \alpha, A, a \right)$ or $\left( \bar{f}_{\alpha}, \bar{A}, \bar{a} \right) = \left( -\left( f - \alpha \right), -A, -a \right)$.
Moreover, we observe that (\ref{eq:constant-dependent-alpha-gamma}) depends on (\ref{eq:constant-dependent-gamma}). As long as there exists $\gamma$ such that (\ref{eq:constant-dependent-gamma}) holds, one can choose $\alpha$ to be small enough (when $\bar{f}_{\alpha} = f - \alpha$) or large enough (when $\bar{f}_{\alpha} = -\left( f - \alpha \right)$) so that (\ref{eq:constant-dependent-alpha-gamma}) follows immediately. In the next lemma, we show that the existence of $\gamma$ satisfying (\ref{eq:constant-dependent-gamma}) can be guaranteed by ``$\bar{a} \in \mathcal{R}(\bar{A})$'' in (i), and hence the problem for the existence of $\alpha,~\gamma$ can be reduced to checking the following conditions (B1)-(B3).

\begin{lemma} \label{lem:liner_sep_checkable}
Let $h(x) = c^T x + c_0$ be an affine function and $f(x) = x^T A x + 2a^T x + a_0$ be a quadratic function. The following statements are equivalent:
\begin{enumerate}[(a)]
\setlength{\itemsep}{6pt}
\item The level set $\{h=\gamma\}$ separates the level set $\{f=\alpha\}$ for some $\alpha, \gamma \in \mathbb{R}$.
\item The matrix $\bar{A} = A$ or $\bar{A} = -A$ satisfies the following three conditions: \vskip4pt
\begin{enumerate}
\setlength{\itemsep}{0.1cm}
\item[(B1)] $\bar{A}$ has exactly one negative eigenvalue, $a \in \mathcal{R}(A)$
\item[(B2)] $c \in \mathcal{R}(A)$, $c \neq 0$
\item[(B3)] $V^T \bar{A} V \succeq 0$
\end{enumerate}
\vskip4pt
where $V \in \mathbb{R}^{n \times (n-1)}$ is a matrix basis of $\mathcal{N}(c^T)$.
\end{enumerate}
\end{lemma}
\begin{proof}
As we say above, according to Lemma \ref{lem:chckable}, the level set $\{h=\gamma\}$ separates $\{f=\alpha\}$ for some $\alpha, \gamma \in \mathbb{R}$ if and only if there exist $\gamma, \alpha$ satisfying conditions (i)-(iii) mentioned above for either $\left( \bar{f}_{\alpha}, \bar{A}, \bar{a} \right) = \left( f - \alpha, A, a \right)$ or $\left( \bar{f}_{\alpha}, \bar{A}, \bar{a} \right) = \left( -\left( f - \alpha \right), -A, -a \right)$.

In either cases, one has
\begin{align}
\bar{a} \in \mathcal{R}(\bar{A}) ~ &\Leftrightarrow ~
a \in \mathcal{R}(A) \label{eq:remove-bar-1} \\
c \in \mathcal{R}(\bar{A}) ~ &\Leftrightarrow ~
c \in \mathcal{R}(A) \label{eq:remove-bar-2}
\end{align}
Therefore, when $\left( \bar{f}_{\alpha}, \bar{A}, \bar{a} \right)$ satisfies (i)-(iii), the same $\bar{A}$ must satisfy conditions (B1)-(B3). Hence, statement (a) directly implies statement (b).

To prove the converse, we take $\bar{a}=a$, $\bar{f}_{\alpha} = f - \alpha$ if $\bar{A} = A$ and take $\bar{a} = -a$, $\bar{f}_{\alpha} = - \left( f - \alpha \right)$ if $\bar{A} = -A$, where $\alpha$ is a constant which will be determined later. With this choice of the triple $(\bar{f}_{\alpha}, \bar{A}, \bar{a})$, the equivalences (\ref{eq:remove-bar-1}) and (\ref{eq:remove-bar-2}) together imply that (i) and (ii) hold true. Finally, it suffices to show the existence of $\alpha, \gamma$ satisfying (iii). We have mentioned that the existence of $\alpha$ depends on the existence of $\gamma$, and then it suffices to show that ``$\bar{a} \in \mathcal{R}(\bar{A})$'' guarantees the existence of $\gamma$ satisfying
\begin{equation} \label{eq:gamma-cond}
\bar{w}_{\gamma} \in \mathcal{R}(V^T \bar{A} V),~\bar{w}_{\gamma} = V^T(\bar{A}x_{\gamma}+\bar{a}),~x_{\gamma} = \frac{-(c_0 - \gamma)}{c^T c} c.
\end{equation}
Since $\bar{a} \in \mathcal{R}(\bar{A})$, there exists some $u_0 \in \mathbb{R}^n$ such that
\begin{equation} \label{eq:V^Ta-in-V^TA}
V^T \bar{a} = V^T \bar{A} u_0.
\end{equation}
As $V \in \mathbb{R}^{n \times (n-1)}$ is a matrix basis for $\mathcal{N}(c^T)$, the $n \times n$ matrix $\left({V}, \frac{1}{{c}^{T} {c}} {c}\right)$ is of full rank. Thus, for $u_0$ in (\ref{eq:V^Ta-in-V^TA}), there exists $y \in \mathbb{R}^{n-1}$ and $\gamma_0 \in \mathbb{R}$ such that
\begin{equation} \label{eq:u0}
u_0 =
\left(V, \frac{1}{c^{T} c} c\right)
\left(\begin{array}{c}
y \\
\gamma_0
\end{array}\right) = V y + \frac{\gamma_0}{c^T c} c.
\end{equation}
Take $\gamma = c_0 - \gamma_0$. Equations (\ref{eq:V^Ta-in-V^TA}) and (\ref{eq:u0}) thus imply
\begin{equation*}
\begin{split}
V^T \bar{a} &= V^T \bar{A} V y + \frac{\gamma_0}{c^T c} V^T \bar{A} c \\
&= V^T \bar{A} V y + \frac{(c_0 - \gamma)}{c^T c} V^T \bar{A} c \\
&= V^T \bar{A} V y - V^T \bar{A} x_{\gamma}
\end{split}
\end{equation*}
Hence, we obtain
$$
V^T \bar{A} V y = V^T \bar{a} + V^T \bar{A} x_{\gamma} = \bar{w}_{\gamma},
$$
which shows that such taken $\gamma$ satisfies (\ref{eq:gamma-cond}), and hence the existence of $\alpha$ follows. The proof is therefore completed.
\end{proof}

Combining Corollary \ref{cor:noncvx_linear_sep} with Lemma \ref{lem:liner_sep_checkable} all together, we see that, if $A \neq 0$, the joint range $\mathbf{R}(f,g)$ is non-convex when and only when the following two parts are satisfied:
\begin{itemize}
\item[-] $A$ and $B$ are linearly dependent, say $B = \lambda A$;
\item[-] two constants $\alpha,~\gamma$ can be chosen so that the hyperplane $\{-\lambda f + g = \gamma \}$ separates the quadratic hypersurface $\{f = \alpha\},$ which can be checked by (B1)-(B3),
\end{itemize}
We write this into the following main theorem, which can be converted into a numerical procedure for checking the convexity of $\mathbf{R}(f,g)$.

\begin{theorem} \label{thm:checkable_cond}
Given two quadratic functions $f(x) = x^T A x + 2 a^T x + a_0$ and $g(x) = x^T B x + 2 b^T x + b_0$ with $A \neq 0$. The joint numerical range $\mathbf{R}(f,g)$ is non-convex if and only if the matrix $\bar{A} = A$ or $\bar{A} = -A$ satisfies the following four:
\begin{enumerate}
\setlength{\itemsep}{0.1cm}
\item[(B0)] $B = \lambda A$ for some $\lambda \in \mathbb{R}$
\item[(B1)] $\bar{A}$ has exactly one negative eigenvalue, $a \in \mathcal{R}(A)$
\item[(B2)] $ -\lambda a + b \in \mathcal{R}(A)$, $-\lambda a + b \neq 0$
\item[(B3)] $V^T \bar{A} V \succeq 0$
\end{enumerate}
where $V \in \mathbb{R}^{n \times (n-1)}$ is a matrix basis of $\mathcal{N}((-\lambda a + b)^T)$.
\end{theorem}

The procedure is described below by a few steps. Notice that each step can be implemented in polynomial time. An implementable pseudo-code is also provided in Algorithm \ref{algorithm} in the next page.


\begin{description}
\setlength{\itemsep}{6pt}
\item[\textbf{Given }] The coefficient matrices $A$, $B$, $a$, and $b$. \vskip0.2cm
\item[\textbf{Step 0}] Check whether $A = B = 0$ or $a = b = 0$.
\vspace{2pt}
\begin{itemize}
\setlength{\itemsep}{2pt}
\item[-] If true, then $\mathbf{R}(f,g)$ is convex.
\item[-] If false, go to next step. (Without loss of generality, we assume $A \neq 0$ in the following.)
\end{itemize}
\item[\textbf{Step 1}] Check whether $B = \lambda A$ for some $\lambda$ in $\mathbb{R}$.
\vspace{2pt}
\begin{itemize}
\setlength{\itemsep}{2pt}
\item[-] If true, set $c = -\lambda a + b$ and go to next step.
\item[-] If false, then $\mathbf{R}(f,g)$ is convex.
\end{itemize}
\item[\textbf{Step 2}] Check whether $c \neq 0$ and whether both the linear systems $Ay_1=a$ and $Ay_2=c$ have solutions.
\vspace{2pt}
\begin{itemize}
\setlength{\itemsep}{2pt}
\item[-] If true, set $V \in \mathbb{R}^{n \times (n-1)}$ to be the matrix basis of $\mathcal{N}(c^T)$ and go to next step.
\item[-] If false, then $\mathbf{R}(f,g)$ is convex.
\end{itemize}
\item[\textbf{Step 3}] Check whether one of the following cases happens:
\begin{align*}
\textrm{(a)} ~ &A \text{ has exactly one negative eigenvalue and } V^T A V \succeq 0 \\
\textrm{(b)} ~ &A \text{ has exactly one positive eigenvalue and } V^T A V \preceq 0
\end{align*}
\begin{itemize}
\setlength{\itemsep}{2pt}
\item[-] If one of (a) and (b) holds, then $\mathbf{R}(f,g)$ is non-convex.
\item[-] Otherwise, $\mathbf{R}(f,g)$ is convex.
\end{itemize}
\end{description}

\begin{algorithm}
\SetKwInOut{Input}{input}\SetKwInOut{Output}{output}

\Input{The coefficient matrices $A$, $B$, $a$, $b$}
\Output{Convexity or non-convexity of $\mathbf{R}(f,g)$}
\BlankLine
\BlankLine

\uIf{$A = B = 0~$ {\rm \bf or} $~a = b = 0$}{
\Return{convex} \;
}
\ElseIf{$A = 0~$ {\rm \bf and} $~B \neq 0$}{
$A \leftarrow B$ \tcc*[c]{Always set $A$ to be the non-zero matrix}
$B \leftarrow 0$ \;
}
\BlankLine

\If{$B = \lambda A$ for some $\lambda \in \mathbb{R}$}{
$c \leftarrow -\lambda a + b$\;
\If{$c \neq 0$}{
	\If{Both linear systems $Ay_1=a$ and $Ay_2=c$ have solutions}
	{
		$V \leftarrow$ matrix basis for null space of $c^T$
		\tcc*[c]{$V$ is an $n \times (n-1)$ matrix}
		\uIf{$A$ has exactly one negative eigenvalue {\rm \bf and} $V^T A V \succeq 0$}{
			\Return{non-convex}\; 		
		}
		\ElseIf{$A$ has exactly one positive eigenvalue {\rm \bf and} $V^T A V \preceq 0$}{
			\Return{non-convex}\;
		}
	}
}
}

\BlankLine
\Return{convex}\;
\BlankLine
\BlankLine
\caption{Algorithm for checking the convexity of $\mathbf{R}(f,g)$}
\label{algorithm}
\end{algorithm}

\begin{example} \label{ex:checkable_computation}
Let $f(x) = x^T A x + 2a^T x$ and $g(x) = x^T B x + 2b^T x$ be two quadratic functions on $\mathbb{R}^n$ with coefficient matrices
\begin{equation*}
\begin{split}
&A =
\left[
\begin{array}{c c c c}
0 & 0 & -1 & 0 \\
0 & \frac{1}{2} & 0 & \frac{1}{2} \\
-1 & 0 & 0 & 0 \\
0 & \frac{1}{2} & 0 & \frac{1}{2}
\end{array}
\right] \ , \
B =
\left[
\begin{array}{c c c c}
 0 & 0 & -2 & 0 \\
 0 & 1 & 0  & 1 \\
-2 & 0 & 0  & 0 \\
 0 & 1 & 0  & 1
\end{array}
\right] \\
&2a = \left( 2, \sqrt{2}, 2, \sqrt{2} \right)^T , 2b = \left( 5, \frac{3}{\sqrt{2}}, 8, \frac{3}{\sqrt{2}} \right)^T
\end{split}
\end{equation*}
\begin{description}
\setlength{\itemsep}{4pt}
\item[\textbf{Step 0}] Neither $A = B = 0$ nor $a = b = 0$.
\item[\textbf{Step 1}] Observe that $B = \lambda A$ with $\lambda = 2$. Set
$$c^T = - \lambda a + b = \left( 1, \frac{-1}{\sqrt{2}}, 4, \frac{-1}{\sqrt{2}} \right)$$
\item[\textbf{Step 2}] One has $c \neq 0$. Both systems $Ay_1=a$ and $Ay_2=c$ have solutions:
$$
y_1 = \left[
\begin{array}{c}
-1 \\
\frac{1}{\sqrt{2}} \\
-1 \\
\frac{1}{\sqrt{2}}
\end{array}
\right] ~ \text{ and }~
y_2 = \left[
\begin{array}{c}
-2 \\
\frac{-1}{2\sqrt{2}} \\
\frac{-1}{2} \\
\frac{-1}{2\sqrt{2}}
\end{array}
\right]
$$
\item[\textbf{Step 3}] Four eigenvalues of $A$ are $\lambda_1 = -1,\ \lambda_2 = 0,\ \lambda_3 = \lambda_4 = 1.$
Thus, $A$ has exactly one negative eigenvalue. We may choose $V \in \mathbb{R}^{n \times (n-1)}$ be the matrix basis of $\mathcal{N}(c^T)$ such that
$$
V =
\left[
\begin{array}{c c c}
\frac{1}{\sqrt{2}} & -4 & \frac{1}{\sqrt{2}} \\
1 & 0 & 0 \\
0 & 1 & 0 \\
0 & 0 & 1
\end{array}
\right]
~ \text{ and } ~
V^T A V = \left[
\begin{array}{c c c}
 \frac{1}{2} & \frac{-1}{\sqrt{2}} & \frac{1}{2} \\
 \frac{-1}{\sqrt{2}} & 8 & \frac{-1}{\sqrt{2}} \\
 \frac{1}{2} & \frac{-1}{\sqrt{2}} & \frac{1}{2} \\
\end{array}
\right].
$$
The matrix $V^T A V$ has eigenvalues $\eta_1 = 0,\ \eta_2 = \frac{9-\sqrt{53}}{2},\ \eta_3 = \frac{9+\sqrt{53}}{2}$, which are all non-negative. Hence, $V^T A V$ is positive semidefinite. Therefore, we conclude that the joint range $\mathbf{R}(f,g)$ is non-convex.
\end{description}
The shaded region in Figure \ref{fig:ex-checkable_computation} is $\mathbf{R}(f,g)$ of this example, which justifies our procedure.

\begin{figure}[H]
\centering
\includegraphics[width=0.6\linewidth]{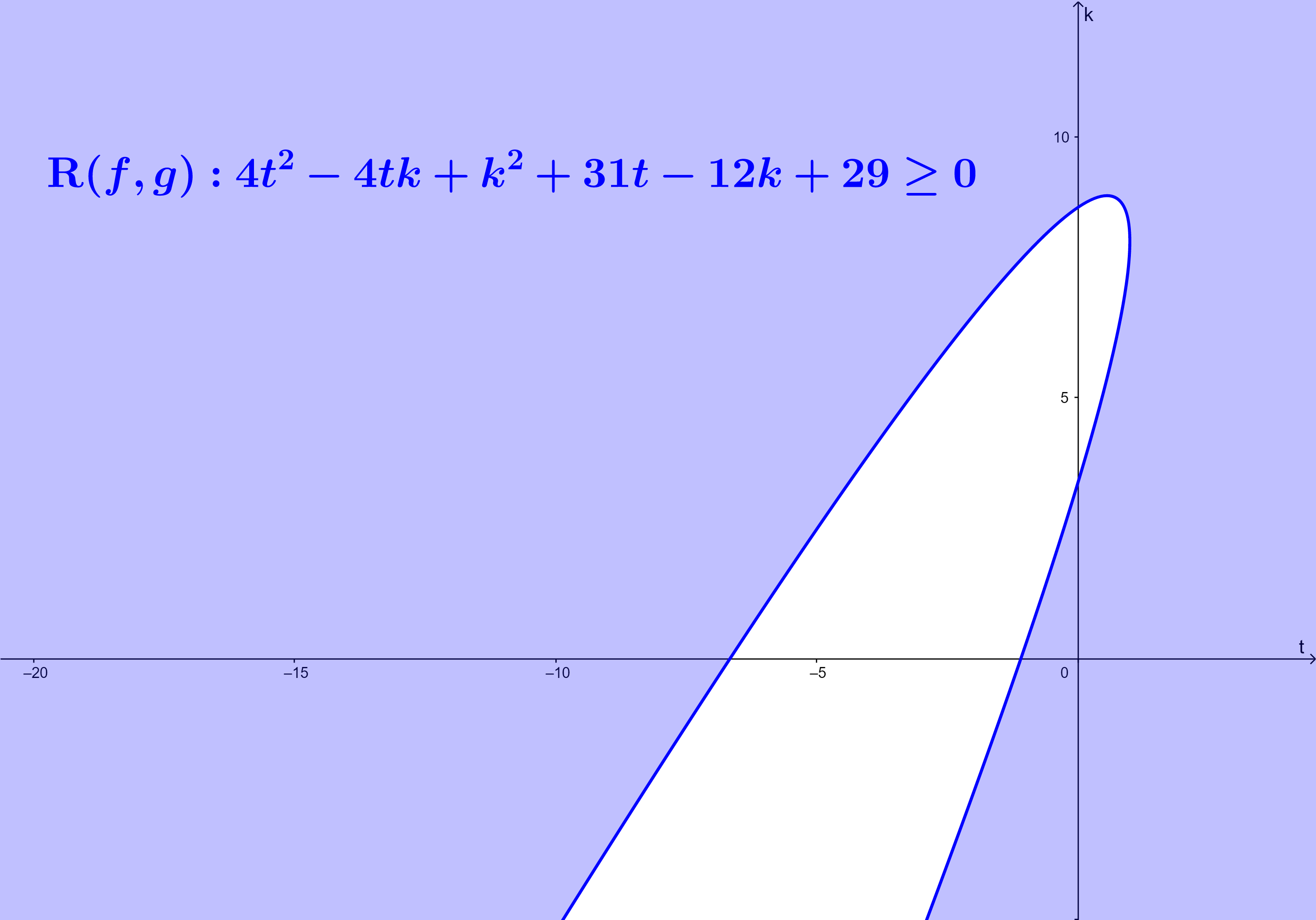}
\caption{The joint numerical range $\mathbf{R}(f,g)$ in Example \ref{ex:checkable_computation}.}
\label{fig:ex-checkable_computation}
\end{figure}
\end{example}

\section{Implications} \label{sec:equiv}

In this section, we show that the necessary and sufficient conditions for the convexity of $\mathbf{R}(f,g)$ developed by Flores-Bazán and Opazo \cite[2016]{Bazan-Opazo16} is a direct consequence of our Theorem \ref{thm:noncvx_sep}. For convenience, we list their result as Theorem \ref{thm:Chile} below.

\begin{theorem}[Flores-Bazán and Opazo {\cite[2016]{Bazan-Opazo16}}] \label{thm:Chile}
The joint numerical range $\mathbf{R}(f,g)$ is non-convex if and only if there exists some $d = (d_1, d_2) \in \mathbb{R}^2$, $d \neq 0$ such that the following four hold:
\begin{enumerate}
\setlength{\itemsep}{4pt}
\item[(C1)] $a,~b \in \left( \mathcal{N}(A) \cap \mathcal{N}(B) \right)^{\perp}$
\item[(C2)] $d_2 A = d_1 B$
\item[(C3)] $-d \in \mathbf{R}(f_H,g_H)$
\item[(C4)] $F_H(u) = -d \implies \left\langle F_L(u), d_{\perp}\right\rangle \neq 0$
\end{enumerate}
where $F_H(x)= \left( f_H(x) , g_H(x) \right) = \left( x^T A x , x^T B x \right)$, $F_L(x)=\left( a^T x , b^T x \right)$, and $d_{\perp} = (-d_2, d_1)$.
\end{theorem}

As we can see, Theorem \ref{thm:Chile} verifies whether $\mathbf{R}(f,g)$ is non-convex by the existence of a certificate $d = (d_1, d_2)\not=0$ satisfying conditions (C1)-(C4), but Flores-Bazán and Opazo in \cite[2016]{Bazan-Opazo16} did not provide a procedure for providing such a certificate. In addition, even with a certificate $d = (d_1, d_2)\not=0$ on hand, conditions (C1)-(C4) reveal very little information as to what was going on behind the scenes.

Our approach reduces the non-convexity of $\mathbf{R}(f,g)$ to also checking the existence of two constants $\alpha,\beta$ such that $\{f=\alpha\}$ separates
$\{g=\beta\}$ or reversely. Then, from Theorem \ref{thm:checkable_cond}, there are four possibilities which could happen:
\begin{itemize}
  \item[($\sharp 1$)] $A$ has exactly one negative eigenvalue ($\{g=\beta\}$ separates $\{f=\alpha\}$);
  \item[($\sharp 2$)] $A$ has exactly one positive eigenvalue ($\{g=\beta\}$ separates $\{f=\alpha\}$);
  \item[($\flat 1$)] $B$ has exactly one negative eigenvalue ($\{f=\alpha\}$ separates $\{g=\beta\}$);
  \item[($\flat 2$)] $B$ has exactly one positive eigenvalue ($\{f=\alpha\}$ separates $\{g=\beta\}$).
\end{itemize}
In the following, we can show that, if $\mathbf{R}(f,g)$ is non-convex, condition (C3) in Theorem \ref{thm:Chile}
can be strengthened to conclude that there exists either a positive eigenvalue or a negative eigenvalue for the $A$ matrix or the $B$ matrix, while there is only one positive (or only one negative) eigenvalue can be derived from condition (C4). In other words, our Theorem \ref{thm:noncvx_sep}, or equivalently, Theorem \ref{thm:checkable_cond} provide more detail information than (C1)-(C4) did, and thus the latter can be put as a consequence of the former. Though we know the two are actually equivalent, yet coming back from (C1)-(C4) to conclude our separation property Theorem \ref{thm:noncvx_sep} is perhaps non-trivial.
%
%
%
%
%

\begin{theorem} \label{thm:implication}
Let $f(x) = x^T A x + 2a^T x + a_0$ and $g(x) = x^T B x + 2b^T x + b_0$ be two quadratic functions defined on $\mathbb{R}^n$ and $B=\lambda A$ for some $\lambda\in \mathbb{R}.$ If the level set $\{g=\beta\}$ separates $\{f=\alpha\}$ for some $\alpha, \beta \in \mathbb{R}$, then $d_{+} = (1, \lambda) \in \mathbb{R}^2$ or $d_{-} = (-1, -\lambda) \in \mathbb{R}^2$ satisfy Conditions (C1)-(C4) in Theorem \ref{thm:Chile}. 
\end{theorem}
\begin{proof}
Since $B=\lambda A$, Condition (C2) holds for both $d_{+} = (1,\lambda)$ and $d_{-} = (-1,-\lambda)$. When $\{g=\beta\}$ separates $\{f=\alpha\}$ for some $\alpha, \beta \in \mathbb{R}$, the function $f$ cannot be affine, and hence $A \neq 0$. Also, according to Lemma \ref{lem:G=kF+H}, when $B=\lambda A$ and $\{g=\beta\}$ separates $\{f=\alpha\}$, the hyperplane $\{-\lambda ( f - \alpha) + (g - \beta) = 0\}$ also separates $\{f=\alpha\}$. Hence, Lemma \ref{lem:liner_sep_checkable} ensures the following three conditions hold:
\begin{enumerate}
\setlength{\itemsep}{0.1cm}
\item[(B1)] $A$ has exactly one negative (resp. positive) eigenvalue, $a \in \mathcal{R}(A)$
\item[(B2)] $c = -\lambda a + b \in \mathcal{R}(A)$, $c \neq 0$
\item[(B3)] $V^T A V \succeq 0$ (resp. $\preceq 0$)
\end{enumerate}
where $V \in \mathbb{R}^{n \times (n-1)}$ is a matrix basis of $\mathcal{N}(c^T)$. In the following, we will show that Conditions (B1)-(B3) imply that $d_{+} = (1,\lambda)$ or $d_{-} = (-1,-\lambda)$ satisfies Conditions (C1), (C3), and (C4).

Note that (C1) is independent to the choice of $d$, so we verify it first. Since $A$ is symmetric, we obtain
\begin{equation*}
\mathcal{R}(A) = \mathcal{R}(A^T) = \mathcal{N}(A)^{\perp}.
\end{equation*}
Also, when $B = \lambda A$, we have $\mathcal{N}(A) \cap \mathcal{N}(B) = \mathcal{N}(A),$ and hence
\begin{equation*} \label{eq:R(A)=N(A)perp}
\mathcal{R}(A) = \left( \mathcal{N}(A) \cap \mathcal{N}(B) \right)^{\perp}.
\end{equation*}
Thus, (B1) and (B3) imply that $a,~b \in \mathcal{R}(A) = \left( \mathcal{N}(A) \cap \mathcal{N}(B) \right)^{\perp}$, which means (C1) holds. For Conditions (C3) and (C4), we divide into two cases according Conditions (B1) and (B3):

When $A$ has exactly one negative eigenvalue and $V^T A V \succeq 0$, we are going to show that $d_{+} = (1,\lambda)$ satisfies (C3) and (C4). For (C3), since $A$ has one negative eigenvalue, there exists $u \in \mathbb{R}^n$ such that $u^T A u = -1$, and hence $u^T B u = -\lambda$ due to relation $B = \lambda A$. Therefore, $F_H(u) = -d_{+}$, which means $d_{+}$ satisfies (C3). To show Condition (C4) holds for $d_{+} = (1,\lambda)$, it suffices to show the following implications:
\begin{equation} \label{eq:pf-C4-equiv}
\left\langle F_L(u), (d_{+})_{\perp}\right\rangle = 0 ~ \implies ~ F_H(u) \neq -d_{+}.
\end{equation}
Observe that
$$
\left\langle F_L(u), (d_{+})_{\perp}\right\rangle = -\lambda a^T u + b^T u = c^T u.
$$
Now, for any $u \in \mathbb{R}^n$ such that $\left\langle F_L(u), (d_{+})_{\perp}\right\rangle = 0$, one has $u \in \mathcal{N}(c^T)$. Then there exists some $w \in \mathbb{R}^{n-1}$ such that $u = V w$. Since $V^T A V \succeq 0$, we have
$$u^T A u = w^T V^T A V w \geq 0,$$
which implies that $u^T A u \neq -1$, and hence $F_H(u) \neq -d_{+}$. Therefore, (\ref{eq:pf-C4-equiv}) holds, which means $d_{+}$ satisfies Condition (C4).

When $A$ has exactly one positive eigenvalue and $V^T A V \preceq 0$, similar argument will ensure that $d_{-} = (-1,-\lambda)$ satisfies Conditions (C3) and (C4).
\end{proof}
\section{Conclusion}
In this paper, we convert the problem about the convexity of $\mathbf{R}(f,g)$ in codomain into the separation property of level sets in the domain. The geometric feature for the non-convexity of the joint numerical range of two quadratic functions $f$ and $g$ is that there exists a pair of level sets $\{g=\beta\}$ and $\{f=\alpha\}$ such that $\{g=\beta\}$ separates $\{f=\alpha\}$ or $\{f=\alpha\}$ separates $\{g=\beta\}$. The result also suggests that S-lemma with equality by Xia et al. \cite[2016]{Xia-Wang-Sheu16} is also a direct consequence of the convexity of $\mathbf{R}(f,g)$. By Nguyen and Sheu \cite[2019]{Quang-Sheu19},
under Slater condition, the S-lemma with equality fails if and only if $\{g=0\}$ separates $\{f<0\}$. Hence, $f$ must have exactly one negative eigenvalue and
thus $\{g=0\}$ separates $\{f=-1\}.$ Finally, our approach also lends itself to a polynomial time procedure for checking the convexity of $\mathbf{R}(f,g),$ which we believe to facilitate more applications in the future.

\section*{Acknowledgement}
Huu-Quang, Nguyen's research work was supported by Taiwan MOST 108-2811-M-006-537 and Ruey-Lin Sheu's research work
was sponsored by Taiwan MOST 107-2115-M-006-011-MY2.

\end{document}